\documentclass[11pt,english,reqno]{amsart}
\usepackage[T1]{fontenc}
\usepackage[latin9]{inputenc}
\usepackage{a4wide}
\usepackage{verbatim} %paquete comentarios
\usepackage{amsthm}
\usepackage{amssymb}
\usepackage{setspace}
\usepackage{enumerate}
\usepackage{fancyhdr}
\usepackage{xcolor}

\usepackage[colorlinks=true]{hyperref}
\hypersetup{linkcolor=red,citecolor=blue,filecolor=dullmagenta,urlcolor=blue}

\numberwithin{equation}{section}
\theoremstyle{plain}
\newtheorem{thm}{Theorem}[section]
\newtheorem{lem}[thm]{Lemma}
\newtheorem{prop}[thm]{Proposition}
\newtheorem{cor}[thm]{Corollary}

\theoremstyle{definition}

\theoremstyle{remark}
\newtheorem{rem}{Remark}[thm]

\newcommand{\R}{\mathbb{R}}

\newcommand{\be}{\begin{equation}}
\newcommand{\ee}{\end{equation}}
\newcommand{\bp}{\begin{proof}}
\newcommand{\ep}{\end{proof}}
\newcommand{\bel}{\begin{equation}\label}
\newcommand{\eeq}{\end{equation}}

\newcommand{\II}{\mathcal{I}_S(t)}

\newcommand{\JJ}{\mathcal{I}_{AN}(t)}

\newcommand{\Hs}{\mathcal{H}_{S}(t)}
\newcommand{\Han}{\mathcal{H}_{AN}(t)}
\newcommand{\Ks}{\mathcal{K}_{S}(t)}
\newcommand{\Ps}{\mathcal{P}_{S}(t)}
\newcommand{\Ran}{\mathcal{R}_{AN}(t)}
\newcommand{\Man}{\mathcal{M}_{AN}(t)}

\makeatother

\usepackage{babel}
\addto\shorthandsspanish{\spanishdeactivate{~<>}}

\addto\captionsenglish{}
\addto\captionsenglish{}
\addto\captionsenglish{}
\addto\captionsenglish{}
\addto\captionsenglish{}
\addto\captionsenglish{}
\addto\captionsenglish{}
\addto\captionsspanish{}
\addto\captionsspanish{}
\addto\captionsspanish{}
\addto\captionsspanish{}
\addto\captionsspanish{}
\addto\captionsspanish{}
\addto\captionsspanish{}

\begin{document}
	\title[Decay for Skyrme wave maps]{Decay for Skyrme wave maps}

\author[Miguel A. Alejo]{Miguel A. Alejo}
\address{Departamento de Matem\'aticas. Universidad de C\'ordoba\\
C\'ordoba, Spain.}
\email{malejo@uco.es}
%\thanks{$^{*}$ M.A.  Alejo was partially supported by  CNPq grant no. 305205/2016-1.}

	\author[Christopher Maul\'en]{Christopher Maul\'en} 
	\address{Departamento de Ingenier\'{\i}a Matem\'atica and Centro
de Modelamiento Matem\'atico (UMI 2807 CNRS), Universidad de Chile, Casilla
170 Correo 3, Santiago, Chile.}
	\email{cmaulen@dim.uchile.cl}
	\thanks{Ch. Maul\'en was partially funded by Chilean research grants FONDECYT 1150202 
	and CONICYT PFCHA/DOCTORADO NACIONAL/2016-21160593. Part of this work was done while the second author was 
	visiting the IMUS (Instituto de Matem\'aticas de la Universidad de Sevilla), Spain.}


	\begin{abstract}
	We consider the decay problem for global solutions of the Skyrme and Adkins-Nappi equations. We prove that the energy associated to any bounded energy solution of the Skyrme (or Adkins-Nappi) equation decays to zero outside the light cone (in the radial coordinate). Furthermore, we prove that suitable polynomial weighted energies of any small solution decays to zero when these energies are bounded. The proof consists of finding three new virial type estimates, one for the exterior of the light cone, based on the energy of the solution, and a more subtle virial identity for the weighted energies, based on a modification of momentum type quantities.
	\end{abstract}
	\maketitle
%\today
\tableofcontents

	\section{Introduction}
This work is concerned with decay properties of global solutions of two \emph{nonlinear} quantum field models, also known in the literature as \emph{Skyrme and Adkins-Nappi equations}. 
Physically these models intend to describe interactions between nucleons and $\pi$ mesons. 
Classical nonlinear field theories played an important role in the description of particles as solitonic objects. A well known example of these 
nonlinear theories is the $SU(2)$ sigma model \cite{Gell}, obtained as a formal critical point from the action

\be\label{su2}
\mathcal{S}(\psi)=\int_{\R^{1,d}}\eta^{\mu\nu}(\psi^*g)_{\mu\nu}=\int_{\R^{1,d}}\eta^{\mu\nu}\partial_\mu \psi^A\partial_\nu \psi^B g_{AB} \circ \psi.
\ee

\noindent
Here $\psi$ is a map from a $(1+d)$-dimensional Minkowski space $(\R^{1,d},\eta)$ to a Riemannian
manifold $(M, g)$ with metric $g$. From a geometrical point of view, the associated Lagrangian is the trace of the
pull-back of the metric $g$ under the map $\psi$. A current choice is $M = \mathbb{S}^d$ with $g$ the associated metric 
and for $d = 3$, one obtains the classical $SU(2)$ sigma model. The Euler-Lagrange equation corresponding to
the action $\mathcal{S}$ is called the wave maps equation. Unfortunately, the $SU(2)$ sigma model
does not admit solitons and it develops singularities in finite time \cite{Bizon, Donninger, Satah}. To avoid these inconveniences and 
to prevent the possible breakdown of the system in finite time, Skyrme \cite{Skyrme} 
modified the associated Lagrangian to \eqref{su2} by adding higher-order terms
which break the scaling invariance of the initial model, which in spherical coordinates 
$(t, r, \theta,\varphi)$ on $\R^{1,3}$, and  co-rotational maps $\psi(t, r, \theta,\varphi)=(u(t,r), \theta,\varphi)$, 
the Skyrme model leads to a scalar quasilinear wave equation satisfied by the angular variable $u$, as it  will be shown in \eqref{skyrme}.
%\be%\label{skyrme}
%	\left( 1+\dfrac{2 \alpha^2 \sin^2(u)}{r^2} \right)(u_{tt}-u_{rr})-
%	\dfrac{2}{r}u_r + \dfrac{ \sin(2u)}{r^2} \left[1+\alpha^2\left(u^2_t -u_r^2 +\dfrac{\sin^2(u)}{r^2} \right) \right]=0.
%\ee

This equation has a \emph{unique} static solution with  boundary values $u(0)=0$ and $\lim_{r\rightarrow\infty}u(r)=\pi,$ and which is 
currently known as \emph{Skyrmion} \cite{Macc}. This existence was proved in \cite{Lad} and \cite{Macc} by using  variational methods and ODE techniques, 
respectively. As far as we know, the Skyrmion is not known in a closed form. 
 Recently, Lawrie and Rodriguez \cite{LR2019} established the existence, uniqueness, and asymptotic stability of topologically nontrivial stationary solutions for the Adkins-Nappi equation. Furthermore, they showed the stable soliton resolution for this equation, conditional on a certain non-conserved norm remaining bounded throughout the evolution. 

For the Adkins-Nappi equation, Geba and Rajeev  \cite{GR2010a} proved that 
solutions remain continuous at the origin and tend to zero as $(t,r)$ goes to zero. Furthermore, in  \cite{GR2010b}, 
they proved that the energy associated to equivariant solutions does not concentrate. 
Finally, Lawrie \cite{L2015} studied the large data dynamics by proving that there is no type II blow-up in the class of maps 
with topological degree zero. In particular, any degree zero map whose critical norm stays bounded must be global-in-time and scatter to
zero as t tends to infinity. 

For the Skyrme model, Geba and Da Silva  \cite{GDS2013} proved that the energy does not concentrate in the 2+1 
dimensional equivariant model. Recently, the large data global regularity for the equivariant case was studied in \cite{GG2017}, 
proving that this is valid for initial data in $H^s \times H^{s-1}(\R^3)$ with $s > 7/2$. Recently, for the  (3+1)-dimensional case, Li \cite{Li} proved the unconditional global-well possedness in $H_{rad}^4\times H_{rad}^{3}(\R^5)$, 
introducing a new method to set global wellposedness for arbitrarily large initial data.
%to overcome the issue of supercriticality

%In the same line,  Li proved that the problem is global-well possednes in $H_{rad}^4\times H_{rad}^{3}(\R^5)$ (see \cite{Li}).

After that, Geba, Nakanishi, and Rajeev \cite{2011GNR} proved global existence and scattering for small data in critical 
homogeneous Sobolev-Besov space (i.e. $\dot{B}^s_{p,q}$) for the Skyrme and Adkins-Nappi equations. 
In particular, considering the change of variable $u=rv$, 
they showed that the equation obtained to $v$ from Skyrme  is globally well posed in 
$C(\R; \dot{B}_{2,1}^{3/2}\cap L^2(\R^{5}))\cap L^2(\R; \dot{B}_{4,1}^{3/4}\cap \dot{B}_{4,2}^{-3/4}(\R^5))$. 
Analogously,  they proved that the equation obtained to $v$ from Adkins-Nappi model is globally well posed
in $C(\R; \dot{H}^{1}\cap L^2(\R^{5}))\cap L^2(\R; \dot{B}_{4,2}^{1/4}\cap \dot{B}_{4,2}^{-3/4}(\R^5))$. See \cite{2011GNR} 
for further details.

\newpage

	\subsection{Main results} In this paper, we are interested in the long time asymptotics of two relevant mathematical physics models. 
	Firstly, the Skyrme model is written as 
	\medskip
\begin{equation}
	\left( 1+\dfrac{2 \alpha^2 \sin^2(u)}{r^2} \right)(u_{tt}-u_{rr})-
	\dfrac{2}{r}u_r + \dfrac{ \sin(2u)}{r^2} \left[1+\alpha^2\left(u^2_t -u_r^2 +\dfrac{\sin^2(u)}{r^2} \right) \right]=0, \label{skyrme}
\end{equation}

\medskip
\noindent
where $\alpha$ is a positive constant having the dimension of length and which will not have any key role in our results. 
The second model is a short of generalization of supercritical wave maps as it was presented in \cite{AN}. This is a simplified version of 
the Skyrme model \eqref{skyrme} and it is currently known as Adkins-Nappi model\medskip
\begin{equation}
u_{tt}-u_{rr}-\dfrac{2}{r}u_r+\dfrac{ \sin(2u)}{r^2}+
\dfrac{\left( u- \sin (u) \cos (u)\right) \left(1-\cos (2u) \right)}{r^4}=0.\label{AdkinsNappi}
\end{equation}

\medskip
\noindent
These two models have the  following low order conserved quantities (subindices "S" and "AN" for Skyrme and Adkins-Nappi models, respectively)

\begin{align}
	E_{S}[u](t)&=\int_{0}^{\infty} r^2 \left[ \left(1+\dfrac{2 \alpha^2 \sin^2(u)}{r^2} \right) (u_t^2+u^2_r)+ 2\dfrac{\sin^2(u)}{r^2}
	+\dfrac{\alpha^2 \sin^4(u)}{r^4} \right], \label{energy_S}\\
	E_{AN}[u](t)&=\int_{0}^{\infty} r^2 \left[ u_t^2+u^2_r+ 2\dfrac{\sin^2(u)}{r^2}+\dfrac{\left(u- \sin (u) \cos (u) \right)^2 }{r^4} \right]
	\label{energy_AN}.
\end{align}

\medskip
(Here $\int_0^{\infty}$ means $\int_{0}^{\infty}dr$.) 
The energies $E_S$ and $E_{AN}$ are well-defined in the homogeneous Sobolev 
spaces $\dot{H}^{7/4} \cap\dot{H}^1(\R^3)$ and  $\dot{H}^{5/3} \cap\dot{H}^1(\R^3)$, respectively.
%With respect to the Cauchy problem, \eqref{skyrme} is globally well-posed for small data in {\color{blue}$\dot{H}^{5/2}(\R^3)$}, and the corresponding global result for the Adkins-Nappi equation \eqref{AdkinsNappi} 
%holds in $(\dot{B}^{5/2}\times \dot{B}^{3/2})\cap (\dot{H}^{1}\times L^2)(\R^3)$ (see \cite{2011GNR}). For large-data global well-posedness results, Li showed that it holds in $H^4(\R^5)$ for Skyrme \eqref{skyrme} (see \cite{Li}).

%The energy doesn't concentrate and small energy implies global regularity \cite{Dan}

\medskip

We denote by $\mathcal{E}^{X}_n$ the space of all finite energy data of degree $n$, namely
\begin{equation} \label{eq:EXn}
\mathcal{E}_n^{X}=\left\{ (u,u_t) \left| E_{X} [u](t)<\infty, \ \ u_0(0)=0, \ \ \lim_{r\to \infty} u_0(r)= n\pi \right. \right\}, 
\end{equation}
where hereafter $X=S$ refers to the Skyrme model or when $X=AN$ to the Adkins-Nappi model. 
In  what follows, we consider $(u,u_t)\in\mathcal{E}^{X}_0$ and such that is a global solution of \eqref{skyrme} or \eqref{AdkinsNappi}, respectively.
%{\color{blue}\[
%u(t,0)=0  \ \mbox{and } \ \lim_{r\to \infty} u(t,r)=0.
%\]
%}
\medskip

The main goal of this work is to prove that small global solutions with enough regularity of Skyrme \eqref{skyrme} and Adkins-Nappi \eqref{AdkinsNappi}
equations decay to zero in a certain region of the light cone. Furthermore, we also study the decay of an associated 
weighted energy for both equations, and which we need it to analyze their corresponding long time behavior.
%{ \color{red} since there is an apparently different dynamic with respect to the cases previously analyzed 
%in \cite{ACKM,AM,MM}.}%, it is needed adequacy respect to how is applied the viral inequalities..} 
%The main tool introduced in this work  is an adapted virial functional and we will use it to prove that the  corresponding small solution of \eqref{skyrme} or \eqref{AdkinsNappi} is decreasing  and {\color{red} follow idea of the prove}.

\medskip

%The main goal of this work is related to study the long-time behavior of the Skyrme and Adkins-Nappi's equations solutions. Our first result deals with the exterior light-cone decay problem.
 More precisely, let $b>0$ and consider the following time depending subset 
\begin{align}\label{S}
R(t)=\{ x\in\R^3\ |\  |x|>(1+b)t \}\subset \R^3,
\end{align}
which is the complement of the ball of radius $(1+b)t$, for $b>0$. We will show that any global solution $u$ to \eqref{skyrme}  (or \eqref{AdkinsNappi}), which is sufficiently regular and without a previous smallness condition,
 must be concentrated inside the light cone. Namely

\begin{thm}[Decay in exterior light cones for the Skyrme and Adkins-Nappi models]\label{thm:S_exterior} \label{thm:AN_exterior}
%The energy asocciated $E_{X}$ to any global solution $u$ of  \eqref{skyrme}  or \eqref{AdkinsNappi}, wich belongs to the space $ \mathcal{E}_0^{X}
	~{} Let $(u_0,u_1) \in \mathcal{E}_0^{X}$, defined in \eqref{eq:EXn}, such that $u$ is a global solution, for \eqref{skyrme} when $X=S$, or \eqref{AdkinsNappi} when $X=AN$, respectively.
	Then, for $R(t)$ as in \eqref{S}, there is strong decay to zero of the energy $E_{X}$, in particular:
	\begin{align}\label{decay0_S_ext} 
		\lim_{t\to\infty} \Vert (u_t(t),u_r(t))\Vert_{L^2\times L^2 (\R^3\cap R(t))}=0.
	\end{align}
	Additionally, one has the mild rate of decay for $|\sigma|>1$:
	\begin{align}\label{integrability0_S_ext} 
		\int_{2}^{\infty} \!\! \int_{0}^{\infty} e^{-c_0\vert r+\sigma t \vert}r^2(u_t^2+u_r^2) drdt \lesssim
		_{c_0} 1.
	\end{align}
\end{thm}
%\begin{rem}
%The condition of smallness is not required
%\end{rem}

\begin{rem}\label{rem:well_defined_spaces}
The spaces $\mathcal{E}_0^{X}$ are not empty. In fact, 
for the Skyrme and Adkins-Nappi equations, the corresponding energies are well-defined in the homogeneous Sobolev 
spaces $\dot{H}^{7/4} \cap\dot{H}^1(\R^3)$ and  $\dot{H}^{5/3} \cap\dot{H}^1(\R^3)$, respectively.
\end{rem}

\medskip

For the next results, we have to introduce a weighted version of the spaces \eqref{eq:EXn}.
Let  $\mathcal{E}^{X,\phi}_n$ the space of all finite $\phi$-weighted energy data of degree $n$
\begin{equation*}
\mathcal{E}_n^{X,\phi}=\left\{ (u,u_t) \left| E_{X,\phi} [u](t)<\infty,\ \ u_0(0)=0, \ \ \lim_{r\to \infty}  u_0(r)=n \pi \right. \right\}, 
\end{equation*}
 where $E_{X,\phi}$ is written for the Skyrme model as
%$E_{S,\phi} [u](t)$,  which in the case $X=S$ is given by 
\begin{equation}
	E_{S,\phi}[u](t)=\int_{0}^{\infty} \phi(r) \left[ \left(1+\dfrac{2 \alpha^2 \sin^2(u)}{r^2} \right) (u_t^2+u^2_r)+ 2\dfrac{\sin^2(u)}{r^2}
	+\dfrac{\alpha^2 \sin^4(u)}{r^4} \right], \label{w_energy_S}\\
\end{equation}	
and for the Adkins-Nappi model as
	%and in the case $X=AN$  
	%has the form
	\begin{equation}
	E_{AN,\phi}[u](t)=\int_{0}^{\infty} \phi(r) \left[ u_t^2+u^2_r+ 2\dfrac{\sin^2(u)}{r^2}+\dfrac{\left(u- \sin (u) \cos (u) \right)^2 }{r^4} \right]
	\label{w_energy_AN}.
\end{equation}
In fact, one can see, that  if $E_{X,r^2}[u](t)=E_{X}[u](t)$, then $\mathcal{E}_n^{X,r^2}=\mathcal{E}_n^{X}$, for $X\in\{S,AN\}$.\\
%both the Skyrme and the Adkins Nappi equations.

\medskip

Our second result shows that the energy $E_X$ associated to any global solution 
$(u,u_t)\in \mathcal{E}_{0}^{X,r^n}~\cap~\mathcal{E}_{0}^{X,r^{n-1}}$of \eqref{skyrme} or \eqref{AdkinsNappi},  decays to zero when $t$ goes to infinity. 
This means that for any global solution $u$ which is sufficiently regular and it satisfies a weighted integrability on 
%the radial variable 
$r$, its energy $E_{X,r^n}$ decays to zero when $t$ goes to infinity for both $X=S$ or $X=AN$ cases.

\medskip

\begin{thm}[Decay of weighted energies]\label{thm:S_interior} \label{thm:AN_interior}
	Let $\delta>0$ small enough. Let $(u,u_t) \in \mathcal{E}_{0}^{X,r^n}\cap \mathcal{E}_{0}^{X,r^{n-1}}$ a global solution of \eqref{skyrme} or 
	\eqref{AdkinsNappi}  such that
	\begin{align}\label{smallness0_S_int}
		\sup_{t\in\R} E_{X}[u](t)<\delta,  \ \mbox{ for  } X=AN,S.
	\end{align}
	 Then, the modified energy $E_{X,\varphi}[u](t)$  with $\varphi(r)=r^{n}$ decays to zero, for  $n> 7$  ($X=S$ case) or 
	 for $n\in\left[\frac{3+\sqrt{41}}{2},10\right]$ ($X=AN$ case), respectively. In particular,
	\begin{align}\label{decay0_S_int}
	\lim_{t\to \infty} \|r^{\frac{n-2}{2}}(u_t,u_r)(t)\|_{L^2\times L^2(\R^3)} = \lim_{t\to \infty} E_{X,r^n}(t)=0.
	%	\lim_{t\to\infty} \Vert (u_r(t),u_t(t))\Vert_{L^2\times L^2(I)}=0.
	\end{align}
\end{thm}

The next remark will be useful in the proof of Theorem \ref{thm:S_interior}.
\begin{rem}[\cite{GR2010a,L2015}]\label{rem:energy_Linfty}
Note that finite energy smooth solutions of Skyrme \eqref{skyrme} and Adkins-Nappi \eqref{AdkinsNappi} equations are uniformly bounded  as follows
\begin{equation*}
\| u\|_{L^\infty_{t,x}} \leq C(E_X [u,u_t](0)), \mbox{where } X\in \{S, AN\}, 
\end{equation*}
and $C(s)\to 0$ as $s\to 0$.
\end{rem}

\subsection{Idea of the proof}

%Hablar sobre dos objetivos
\medskip

In order to prove Theorem \ref{thm:S_exterior}, 
we follow some ideas appeared in \cite{ACKM,AM,MM}, 
where decay for Camassa-Holm, Born-Infeld and Improved-Boussinesq models were considered. 
The main tool in these works was a suitable virial functional for which the dynamic of solutions is converging to zero when
 it is integrated in time. 
\medskip

In this paper, the new virial functionals give us relevant information about the dynamics of global solutions of Skyrme and Adkins-Nappi equations. 
Using a proper virial estimate, we prove that the corresponding energies associated to Skyrme and Adkins-Nappi equations decay 
to zero in the subset $R(t)$ \eqref{S}.

\medskip

%Furthermore, to prove Theorem \ref{thm:AN_interior},  we will study the growth rate of polynomial weight energies 
%of the Skyrme and Adkins-Nappi equations. After that, {\color{blue}  assuming that their growth are bounded, we will be prove that decays to zero as t tends to infinity.}\footnote{No entiendo esta frase, explicar mejor} 
%To prove this result, we introduce a functional associated with a sort of weighted momentum. It happens that the  
%virial identity associated to this functional 
%no shows evidence of good sign conditions. 
%does not satisfy \underline{good sign properties}. 
Furthermore, to prove Theorem \ref{thm:AN_interior},  we will study the growth rate of polynomial weight energies of the Skyrme and Adkins-Nappi equations. After that,   assuming that their growth is bounded, we will prove that this 
growth decays to zero as $t$ tends to infinity. To prove this result, we introduce a functional associated with a sort of weighted momentum. 
It happens that the virial identity associated to this functional shows no evidence of good sign conditions, 
i.e. that the derivative of the functional is negative. Therefore, we have to introduce a new functional 
as a linear combination of these two viral identities and for which there is a good sign property. This ensures the integrability 
in time of polynomial weighted energies of degree $n$. Moreover, it also guarantees the decay of a polynomial weighted energy of degree $n+1$ over 
a subsequence of times. Combining these two facts, we conclude that the polynomial weighted energies, which are bounded, decay 
to zero as $t$ tends to infinity (over $\R^3$).
%
%{\color{red}
%For the proof of Theorem \ref{thm:S_exterior}, we used techniques that are not new. Let $u\in \mathcal{E}_{0}^X$ a global solution, i.e. 
%a global solution of degree zero and such that their energy is well-defined, and consider their }}

%For the proof of Theorem \ref{thm:S_interior} We will consider $u \in \mathcal{E}_{0}^{X,\phi}$}

%No evidence of good sign conditions is clearly shown here. This identity is far from being useful 
%(actually, it is the first case among the above mentioned equations where it fails to give decay information). The key to prove decay is an additional term in the virial, called N (t) (see (2.5)-(2.6)), 
%that allows us to recover the positivity of the virial inequality of $\frac{d}{dt}\mathcal{H}_X$:

%Since the localization in the momentum does not give us a good sign which allows concluding.}

\subsection*{Organization of this paper}
 This chapter is organized as follows: Section \ref{sec:virial_identities} is splitted in two subsections where a series of virial identities are presented: in Subsection \ref{Sect:virial_skyrme} and \ref{Sect:virial_AdkinsNappi}  we show the virial identities used for to prove the decay of the energy in the Skyrme  equation \eqref{skyrme} and \eqref{AdkinsNappi}, respectively. 
 Section \ref{Sect:4} deals with the proof of Theorem \ref{thm:S_exterior} for the Skyrme  and Adkins-Nappi equations.
 Finally,  Section \ref{Sect:5}  deals with the proof of Theorem \ref{thm:AN_interior} for both models.

% This paper as organized as follows. Section \ref{sec:virial_identities} is splitted 
%in two subsections where is a series of virial identities are presented: in Subsection \ref{Sect:virial_skyrme} and \ref{Sect:virial_AdkinsNappi} 
%we show the virial identities used for to prove the decay of the energy in the Skyrme  equation \eqref{skyrme} and \eqref{AdkinsNappi}, respectively. Section \ref{Sect:4} deals with the proof of Theorem \ref{thm:S_exterior} for the Skyrme  and Adkins-Nappi equations. Finally, 
%Section \ref{Sect:5}  deals with the proof of Theorem \ref{thm:AN_interior} for the Adkins-Nappi equation.
\subsection*{Acknowledgments} 
The authors are indebted to C. Mu\~noz for stimulating discussions and valuable suggestions that helped to improve a previous version of this work. 
Second author deeply thanks F. Gancedo (U. Sevilla) for his hospitality during some research stays 
where part of this work was done.

\section{Virial Identities}\label{sec:virial_identities}
	In this section three virial identities for the Skyrme and Adkins-Nappi models \eqref{skyrme}-\eqref{AdkinsNappi} are presented.
	One of the virial functionals is related with the exterior light cone behavior (Theorem \ref{thm:S_exterior}), and the other ones 
	are useful for understanding the decay of the weighted energy for Skyrme and Adkins-Nappi models  (Theorem \ref{thm:S_interior}). 
	 Moreover, we remark here that the energies $E_{S}[u]$ and $E_{AN}[u]$, defined in \eqref{energy_S} and \eqref{energy_AN}, 
	 are bounded in spaces $\mathcal{E}_{0}^S$ and $\mathcal{E}_{0}^{AN}$, respectively. 
	Furthermore, 
 it is well-known that these energies are well defined  in the homogeneous  
	Sobolev spaces $\dot{H}^{7/4} \cap\dot{H}^1(\R^3)$ and $\dot{H}^{5/3}\cap \dot{H}^1(\R^3)$ for the Skyrme and Adkins-Nappi equations, respectively.

	\subsection{Virial identities for the Skyrme Model}\label{Sect:virial_skyrme}

%	In this section we present three virial identities for the Skyrme model \eqref{skyrme}.
%	One is related with the exterior light cone behavior (Theorem \ref{thm:S_exterior}), and the others two is useful for understanding the decay of the weighted energy
%	(Theorem \ref{thm:S_interior}).
%
%	\medskip

	Let
%	$\lambda(t)$ a function that has never zero time scaling be a large parameter and
$\varphi=\varphi(t,r)$ be a smooth, bounded weight function, to be chosen later.
For each $t\in \R$ we consider the following functional
	\begin{align}
	\II=& \int_{0}^{\infty} \varphi r^2 \left[ \left(1+\dfrac{2 \alpha^2 \sin^2(u)}{r^2} \right)
	(u_t^2+u^2_r)+ 2\dfrac{\sin^2(u)}{r^2}+\dfrac{\alpha^2 \sin^4(u)}{r^4} \right] \label{I},
	% \varphi\left(\dfrac{r +\sigma t}{L} \right)
	\end{align}
which is a generalization of the energy introduced in \eqref{energy_S}, and well-defined for $(u,u_t)\in (\dot{H}^{7/4} \cap\dot{H}^1) \times L^2 (\R^3)$.
Moreover, if $\varphi$ only depends  on $r$ and it can be written as $\varphi(r)=\phi/r^2$,  then we recover $E_{S,\phi}$, which is the weighted energy defined in \eqref{w_energy_S}.  The following  identities will be useful for the proof of Theorems \ref{thm:S_exterior}-\ref{thm:S_interior}. 

  The following result shows the variation of the localized energy   for the Skyrme equation:
  	\begin{lem}[Energy local variations: Skyrme Model]\label{lem:I'}
	For any $t\in \R,~\varphi(t,r)$ a smooth function previously defined, and  $\II$ as in \eqref{I}, we have that
	\begin{equation}
			\begin{aligned}
				\dfrac{d}{dt}\II=&~{} \int_{0}^{\infty} \varphi_t r^2 \left[ \left(1+\dfrac{2 \alpha^2 \sin^2(u)}{r^2} \right)
				(u_t^2+u^2_r)+ 2\dfrac{\sin^2(u)}{r^2}+\dfrac{\alpha^2 \sin^4(u)}{r^4} \right]\\
				&-\int_{0}^{\infty} \varphi_r r^2  \left(1+\dfrac{2 \alpha^2 \sin^2(u)}{r^2} \right) 2 u_t u_r.\label{derI1}
				\end{aligned}
				\end{equation}
	\end{lem}

	\begin{proof}[Proof of Lemma \ref{lem:I'}] Derivating  \eqref{I} with respect to time,
	and using a basic trigonometric relation, we have
		\[
		\begin{aligned}
			\dfrac{d}{dt}\II
%			=&~{}
%			 {\color{red}\int_{0}^{\infty} \varphi_t r^2 \left[ \left(1+\dfrac{2 \alpha^2 \sin^2(u)}{r^2} \right) (u_t^2+u^2_r)+ 2\dfrac{\sin^2(u)}{r^2}+\dfrac{\alpha^2 \sin^4(u)}{r^4} \right] }\\
%			%
%			& {\color{red}+\int_{0}^{\infty} \varphi r^2 \left[ \left(\dfrac{2 \alpha^2 u_t \sin(2 u)  }{r^2} \right) (u_t^2+u^2_r)
%			+ 2\left(1+\dfrac{2 \alpha^2 \sin^2(u)}{r^2} \right) (u_t u_{tt}+u_r u_{rt})\right]}\\
%			%
%			& {\color{red}+\int_{0}^{\infty} \varphi r^2 \left[2\dfrac{ u_t \sin(2u) }{r^2}+\dfrac{2\alpha^2 u_t \sin^2(u) \sin(2u)   }{r^4} \right] }\\
%			%%
			=&~{}
			 \int_{0}^{\infty} \varphi_t r^2 \left[ \left(1+\dfrac{2 \alpha^2 \sin^2(u)}{r^2} \right) (u_t^2+u^2_r)+ 2\dfrac{\sin^2(u)}{r^2}+\dfrac{\alpha^2 \sin^4(u)}{r^4} \right]\\
			&+2\int_{0}^{\infty} \varphi r^2 u_t\bigg[ \left(\dfrac{ \alpha^2  \sin(2 u)  }{r^2} \right) (u_t^2+u^2_r)
			+\dfrac{  \sin(2u) }{r^2}+\dfrac{\alpha^2 \sin^2(u) \sin(2u)   }{r^4} \bigg]\\
			&+2\int_{0}^{\infty} \varphi r^2 u_t
			\left(1+\dfrac{2 \alpha^2 \sin^2(u)}{r^2} \right) u_{tt} 
			+2\int_{0}^{\infty} \varphi r^2 \left(1+\dfrac{2 \alpha^2 \sin^2(u)}{r^2} \right) u_r u_{rt}\\
			&:=I_1+I_2+I_3+I_4.
		\end{aligned}
		\]
Now, using the equation \eqref{skyrme} in $I_3$, we have
	\[
	\begin{aligned}
		%\int_{0}^{\infty} \varphi r^2 u_t \left(1+\dfrac{2 \alpha^2 \sin^2(u)}{r^2} \right) u_{tt}
		I_3
		=&
		2\int_{0}^{\infty} \varphi r^2 u_t\left\{
		\left( 1+\dfrac{2 \alpha^2 \sin^2(u)}{r^2} \right)u_{rr}+
		\dfrac{2}{r}u_r
		- \dfrac{ \sin(2u)}{r^2} \left[1+\alpha^2\left(u^2_t -u_r^2 +\dfrac{\sin^2(u)}{r^2} \right) \right]\right\}.
	\end{aligned}
	\]
	And integrating by parts in the last integral $I_4$, we obtain
	\[
	\begin{aligned}
			\frac12 I_4
	%		\int_{0}^{\infty} \varphi  \left(r^2+2 \alpha^2 \sin^2(u) \right) u_r u_{rt}\\
			%%%
		% =[r^2+2\alpha^2 \sin^2(u)]u_t u_r\bigg|_{r=0}^{\infty}  \\
		% -\int_{0}^{\infty} \bigg(\varphi  \left(r^2+2 \alpha^2 \sin^2(u) \right) u_t\bigg)_{r} u_{r})\\
		% %%
		% =[r^2+2\alpha^2 \sin^2(u)]u_t u_r\bigg|_{r=0}^{\infty}  \\
		% -\int_{0}^{\infty} \bigg(\varphi_r \left(r^2+2 \alpha^2 \sin^2(u) \right) u_t  +\varphi \bigg( \left(2r+4 \alpha^2 \sin (u) \cos (u) u_r \right) u_t+  \left(r^2+2 \alpha^2 \sin^2(u) \right)  u_{tr}\bigg)\right) \bigg) u_{r})\\
		%%%
		=&~
		%{\color{red}\varphi[r^2+2\alpha^2 \sin^2(u)]u_t u_r\bigg|_{r=0}^{r=\infty} }
		-\int_{0}^{\infty} \varphi_r r^2\left(1+\frac{2 \alpha^2 \sin^2(u)}{r^2} \right)u_r u_t\\
		%sin2u=cosu sinu
		&-\int_{0}^{\infty}\varphi r^2 u_{t} \bigg(\left(\dfrac{2}{r}u_r+\frac{2 \alpha^2 \sin (2u)}{r^2} u_r^2 \right)
		-  \left(1+\dfrac{2 \alpha^2 \sin^2(u)}{r^2} \right)  u_{rr}\bigg) .
	\end{aligned}
	\]
Finally, we have
\[
\begin{aligned}
	I_2+I_3+I_4
	=&~{}
	 -2\int_{0}^{\infty} \varphi_r r^2\left(1+\frac{2 \alpha^2 \sin^2(u)}{r^2} \right) u_t u_r,
\end{aligned}
\]
and we get that
\[
\begin{aligned}
\dfrac{d}{dt}\II
=&~{}
 \int_{0}^{\infty} \varphi_t r^2 \left[ \left(1+\dfrac{2 \alpha^2 \sin^2(u)}{r^2} \right) (u_t^2+u^2_r)+ 2\dfrac{\sin^2(u)}{r^2}
 +\dfrac{\alpha^2 \sin^4(u)}{r^4} \right]\\
&-2\int_{0}^{\infty} \varphi_r r^2\left(1+\frac{2 \alpha^2 \sin^2(u)}{r^2} \right) u_t u_r.
\end{aligned}
\]
%Using that $\varphi=\varphi\left(\frac{r+\sigma t}{L}\right)$ we conclude.
This concludes the proof.
	\end{proof}
\begin{rem}\label{rem:I'}
% for $\varphi=\phi/r^2$ we recover $E_{s,\phi}$
With the change of variables $\varphi=\phi/r^2$, we avoid the term $r^2$ in the weighted function \eqref{I}, 
and which coming from the dimension of the problem. Then,  $E_{S,\phi}$ is recovered from $\II$ and applying the Lemma \ref{lem:I'}, we get 
% % The term $r^2$ , which is given by the dimension of the problem, in the weighted function \eqref{I} could be avoided considering $\varphi=\phi/r^2$. 
% Then  $E_{S,\phi}$ is recovered from $\II$ and applying the Lemma \ref{lem:I'}, we get 
\begin{equation}\label{eq:I_wo_r2}
\begin{aligned}
\dfrac{d}{dt}E_{S,\phi}
=&~{}
 \int_{0}^{\infty} \phi_t \left[ \left(1+\dfrac{2 \alpha^2 \sin^2(u)}{r^2} \right) (u_t^2+u^2_r)+ 2\dfrac{\sin^2(u)}{r^2}
 +\dfrac{\alpha^2 \sin^4(u)}{r^4} \right]\\
&-2\int_{0}^{\infty} \left(\phi' -2\frac{\phi}{r}\right)\left(1+\frac{2 \alpha^2 \sin^2(u)}{r^2} \right) u_t u_r.
\end{aligned}
\end{equation}
This relation will be useful in the proof of Theorem \ref{thm:S_interior}.
\end{rem}

%We define the moment of the inertia of the Skyrme model as
%\begin{equation}
%	\mathcal{M}_S=
%	\int_{0}^\infty r^2 \sin^2(u)\left( 1+u_t^2+\frac{\sin^2(u)}{r^2}\right) 
%\end{equation}
Now, we define two functionals that we will use to prove the decay of the weighted energy $E_{X,\phi}$. Firstly, 
denote by $f$ the following function
\begin{equation}\label{f}
f(u)= 1+\frac{2 \alpha^2 \sin^2(u)}{r^2}.
\end{equation}
%with 
%\begin{equation}\label{f_der}
%\begin{aligned}
%(f(u))_t=2\alpha^2 \frac{\sin(2u)u_t}{r^2}, \quad \mbox{and }\quad (f(u))_r=2\alpha^2\frac{\sin(2u)u_r}{r^2}-4\alpha^2\frac{\sin^2(u)}{r^3}.
%\end{aligned}
%\end{equation}
Now, considering $\psi$ and $\phi$ smooth weight functions of $r$, which will be chosen later, we define the functional 
 $\Ks$ associated with a sort of momentum, given by
\be\label{K}
\begin{aligned}
\Ks=\int_{0}^\infty  \psi f(u)u_t u_r,
%&\Ps_S(t)=\int_{0}^\infty  \varphi r^n\left(1+\frac{2 \alpha^2 \sin^2(u)}{r^2} \right) u_t u 
\end{aligned}
\ee
\noindent 
and  the functional $\Ps$, which corrects the bad sign of the variation in time on the functional $\Ks$, and which is given by
\be\label{P}
\begin{aligned}
%&\Ks(t)=\int_{0}^\infty  \varphi r^n\left(1+\frac{2 \alpha^2 \sin^2(u)}{r^2} \right) u_t u_r\\
\Ps=\int_{0}^\infty  \phi f(u) u_t u.
\end{aligned}
\ee

\begin{lem}[]\label{lem:K'}
Let $t\in \R$, $\psi$ be a smooth weight function and $\Ks$ defined as in \eqref{K}. If $u\in \mathcal{E}_0^{S,\psi}$
 and $p(r)= \left(\frac{\psi'}{r^2} -4\frac{\psi}{r^3}\right)$, then we have
%	For any $t\in \R$ and {\color{red} $\psi(r)\in C^{2}$}, one has
	\begin{equation}\label{eq:dt_KS}
	\begin{aligned}
	\dfrac{d}{dt}\Ks=&~{}
% 			-\frac12 \int_{0}^\infty  \frac{\psi'}{r^2}  r^2 u_t^2 
% 			-\frac12 \int_{0}^\infty  \left(\frac{\psi'}{r^2} -4\frac{\psi}{r^3}\right) r^2 u_r^2\\
% 			&-\frac{\alpha^2}{2} \int_{0}^\infty\left(\frac{\psi'}{r^2} -4\frac{\psi}{r^3}\right)  \sin^2(u) (u_t^2+u_r^2)\\
% 			&
% 			+\int_{0}^\infty  \left(\frac{\psi'}{r^2}-2\frac{\psi}{r^3}\right)\sin^2 (u) 
% 			+\frac{\alpha^2}{2}  \int_{0}^\infty  \left(\frac{\psi'}{r^2}- 4\frac{\psi}{r^3}\right)   \frac{\sin^4 (u)}{r^2} .
			-\frac12 \int_{0}^\infty  \frac{\psi'}{r^2}  r^2 u_t^2 
			-\frac12 \int_{0}^\infty  p(r) r^2 u_r^2 
			-\frac{\alpha^2}{2} \int_{0}^\infty p(r)  \sin^2(u) (u_t^2+u_r^2)\\
			&
			+\int_{0}^\infty  \left(\frac{\psi'}{r^2}-2\frac{\psi}{r^3}\right)\sin^2 (u) 
			+\frac{\alpha^2}{2}  \int_{0}^\infty  p(r)   \frac{\sin^4 (u)}{r^2} .
\end{aligned}
	\end{equation}
\end{lem}

\begin{proof}
First of all, we notice that the time and radial derivates of $f$ \eqref{f} are
%\begin{equation}\label{f}
%f(u)= 1+\frac{2 \alpha^2 \sin^2(u)}{r^2}.
%\end{equation}
%with 
\begin{equation}\label{f_der}
\begin{aligned}
(f(u))_t=2\alpha^2 \frac{\sin(2u)u_t}{r^2}, \quad \mbox{and }\quad (f(u))_r=2\alpha^2\frac{\sin(2u)u_r}{r^2}-4\alpha^2\frac{\sin^2(u)}{r^3}.
\end{aligned}
\end{equation}
Secondly, derivating the functional \eqref{K} with respect to time, we get
	\[
	\begin{aligned}
			\frac{d}{dt} \Ks
			=&\int_{0}^\infty  \psi \left(f(u) \right)_t u_t u_r 
			+\int_{0}^\infty  \psi f(u) (u_{tt} u_r+u_t u_{rt})\\
			=&\int_{0}^\infty  \psi \left(f(u) \right)_t u_t u_r 
			+\int_{0}^\infty  \psi f(u) u_{tt} u_r +\frac12 \int_{0}^\infty  \psi f(u) (u_t^2)_r.
			%%%
	\end{aligned}
	\]
Integrating by parts in the last term of the RHS, we obtain
	\begin{equation}\label{eq:Ks_def}
	\begin{aligned}
				\frac{d}{dt} \Ks
			=&\int_{0}^\infty  \psi \left(f(u) \right)_t u_t u_r 
			+\int_{0}^\infty  \psi f(u) u_{tt} u_r 
			-\frac12 \int_{0}^\infty  \psi' f(u) u_t^2\\
			&-\frac12 \int_{0}^\infty  \psi (f(u))_r u_t^2 
	:=~{}K_1+K_2+K_3+K_4.
	\end{aligned}
	\end{equation}
	For $K_2$, using \eqref{skyrme}, we obtain
	\[
	\begin{aligned}
	K_2
	=&\int_{0}^\infty  \psi u_r \left\{
	f(u) u_{rr}+\frac{2}{r} u_r-\dfrac{ \sin(2u)}{r^2} \left[1+\alpha^2\left(u^2_t -u_r^2 +\dfrac{\sin^2(u)}{r^2} \right)
	\right] \right\}\\
	=%&{}~{\color{red} \psi f(u) \frac{u_{r}^2}{2}\bigg|_{r=0}^{r=\infty}}\\
		&-\int_{0}^\infty  \left(\psi  
	f(u) \right)_r \frac{u_{r}^2}{2}
	+2\int_{0}^\infty  \frac{\psi}{r} u_r^2
	-\int_{0}^\infty  \frac{\psi}{r^2}  \sin (2u) u_r \left[1+\alpha^2\left(u^2_t -u_r^2 +\dfrac{\sin^2(u)}{r^2} \right)
	\right] \\
	=&-\int_{0}^\infty  \left[\psi ' f(u)+\psi  (f(u))_r \right] \frac{u_{r}^2}{2}
	+2\int_{0}^\infty  \frac{\psi}{r} u_r^2\\
	&-\int_{0}^\infty  \frac{\psi}{r^2}  \sin (2u) u_r \left[1+\alpha^2\left(u^2_t -u_r^2 +\dfrac{\sin^2(u)}{r^2} \right)
	\right], \\
%	=&-\int_{0}^\infty 
%	 \varphi_r 
%	\left(r^n+r^{n-2} 2 \alpha^2 \sin^2(u) \right) \frac{u_{r}^2}{2} 
%	-\int_{0}^\infty \varphi   
%	\left(nr^{n-1}+(n-2)r^{n-3} 2 \alpha^2 \sin^2(u)\right)\frac{u_{r}^2}{2}\\	
%	&-\int_{0}^\infty \varphi r^{n-2} \alpha^2 \sin(2u)  u_{r}^3
%	+\int_{0}^\infty  \varphi 2r^{n-1} u_r^2\\
%	&-\int_{0}^\infty  \varphi r^{n-2}  \sin (2u) u_r -\int_{0}^\infty  \varphi r^{n-2}  \alpha^2 \sin (2u) \left(u_r u^2_t -u_r^3  \right) \\
%	&-\int_{0}^\infty  \varphi r^{n-4}  \alpha^2 \sin (2u) u_r \sin^2(u) \\\
%		=&-\int_{0}^\infty 
%	\varphi_r 
%	\left(r^n+r^{n-2} 2 \alpha^2 \sin^2(u) \right) \frac{u_{r}^2}{2} 
%	-\int_{0}^\infty \varphi   
%	\left(nr^{n-1}+(n-2)r^{n-3} 2 \alpha^2 \sin^2(u)\right)\frac{u_{r}^2}{2}\\	
%	&
%	+\int_{0}^\infty  \varphi 2r^{n-1} u_r^2
%	-\int_{0}^\infty  \varphi r^{n-2}  \sin (2u) u_r -\int_{0}^\infty  \varphi r^{n-2}  \alpha^2 \sin (2u) u_r u^2_t  \\
%	&-\int_{0}^\infty  \varphi r^{n-4}  \alpha^2 \sin (2u) u_r \sin^2(u) \.
	\end{aligned}
	\]
	and replacing \eqref{f_der}, we get
		\begin{equation}\label{K2}
	\begin{aligned}
	K_2
%	=&-\int_{0}^\infty  \left[\psi ' f(u)+\psi  (f(u))_r \right] \frac{u_{r}^2}{2}
%	+2\int_{0}^\infty  \frac{\psi}{r} u_r^2\\
%	&-\int_{0}^\infty  \frac{\psi}{r^2}  \sin (2u) u_r \left[1+\alpha^2\left(u^2_t -u_r^2 +\dfrac{\sin^2(u)}{r^2} \right)
%	\right] \\
	=& -\int_{0}^\infty  \psi ' f(u)  \frac{u_{r}^2}{2}
	-\alpha^2\int_{0}^\infty \psi \frac{\sin(2u)}{r^2} u_{r}^3
	+2\alpha^2 \int_{0}^\infty \psi \frac{\sin^2(u)}{r^3}  u_{r}^2
	+2\int_{0}^\infty  \frac{\psi}{r} u_r^2\\
	&-\int_{0}^\infty  \frac{\psi}{r^2}  \sin (2u) u_r \left[1+\alpha^2\left(u^2_t -u_r^2 +\dfrac{\sin^2(u)}{r^2} \right)
	\right] \\
	=& -\int_{0}^\infty  \psi ' f(u)  \frac{u_{r}^2}{2}
	+2\alpha^2 \int_{0}^\infty \psi \frac{\sin^2(u)}{r^3}  u_{r}^2
	+2\int_{0}^\infty  \frac{\psi}{r} u_r^2\\
	&-\int_{0}^\infty  \frac{\psi}{r^2}  \sin (2u) u_r \left[1+\alpha^2\left(u^2_t  +\dfrac{\sin^2(u)}{r^2} \right)
	\right]. 
	\end{aligned}
	\end{equation}
Using \eqref{K2} and \eqref{f_der} in \eqref{eq:Ks_def}, one can see
	\[
	\begin{aligned}
		\frac{d}{dt} \Ks
%			=&
%			\int_{0}^\infty  \psi 2\alpha^2 \frac{\sin(2u)}{r^2} u_t^2 u_r 
%			-\frac12 \int_{0}^\infty  \psi' f(u) u_t^2\\
%			&- \int_{0}^\infty  \psi  \alpha^2\frac{\sin(2u)u_r}{r^2} u_t^2 +2\int_{0}^\infty \psi \alpha^2\frac{\sin^2(u)}{r^3} u_t^2\\
%			&-\int_{0}^\infty  \psi ' f(u)  \frac{u_{r}^2}{2}
%	+2\alpha^2 \int_{0}^\infty \psi \frac{\sin^2(u)}{r^3}  u_{r}^2
%	+2\int_{0}^\infty  \frac{\psi}{r} u_r^2\\
%	&-\int_{0}^\infty  \frac{\psi}{r^2}  \sin (2u) u_r \left[1+\alpha^2\left(u^2_t  +\dfrac{\sin^2(u)}{r^2} \right)
%	\right] \\
%%%	=&
%%%			%\alpha^2 \int_{0}^\infty  \psi  \frac{\sin(2u)}{r^2} u_t^2 u_r 
%%%			-\frac12 \int_{0}^\infty  \psi' \left(  1+\frac{2 \alpha^2 \sin^2(u)}{r^2} \right) (u_t^2+u_r^2)\\
%%%			&%- \alpha^2 \int_{0}^\infty  \psi  \frac{\sin(2u)}{r^2} u_t^2u_r 
%%%			+2\alpha^2\int_{0}^\infty \psi \frac{\sin^2(u)}{r^3} (u_t^2+u_r^2)\\
%%%			&%+2\alpha^2 \int_{0}^\infty \psi \frac{\sin^2(u)}{r^3}  u_{r}^2
%%%			+2\int_{0}^\infty  \frac{\psi}{r} u_r^2\\
%%%			&-\int_{0}^\infty  \frac{\psi}{r^2}  \sin (2u) u_r
%%%			%-\alpha^2 \int_{0}^\infty  \frac{\psi}{r^2}  \sin (2u) u_r u^2_t  
%%%			-\alpha^2 \int_{0}^\infty  \frac{\psi}{r^2}  (\sin^2 (u))_r \dfrac{\sin^2(u)}{r^2}  \\
				=&
			-\frac12 \int_{0}^\infty  \psi' \left(  1+\frac{2 \alpha^2 \sin^2(u)}{r^2} \right) (u_t^2+u_r^2) 
			+2\alpha^2\int_{0}^\infty \psi \frac{\sin^2(u)}{r^3} (u_t^2+u_r^2)\\
			&
			+2\int_{0}^\infty  \frac{\psi}{r} u_r^2
			-\int_{0}^\infty  \frac{\psi}{r^2}  (\sin^2 (u))_r
			-\frac{\alpha^2}{2} \int_{0}^\infty  \frac{\psi}{r^4}  (\sin^4 (u))_r.
	\end{aligned}
	\]
	Finally, integrating by parts and regrouping terms, we get%conclude the proof of the lemma
	%Now, integrating by parts
		\[
	\begin{aligned}
		\frac{d}{dt} \Ks
%%					=&
%%			-\frac12 \int_{0}^\infty  \psi' \left(  1+\frac{2 \alpha^2 \sin^2(u)}{r^2} \right) (u_t^2+u_r^2) 
%%			+2\alpha^2\int_{0}^\infty \frac{\psi}{r} \frac{\sin^2(u)}{r^2} (u_t^2+u_r^2)\\
%%			&
%%			+2\int_{0}^\infty  \frac{\psi}{r} u_r^2
%%			+\int_{0}^\infty  \left(\frac{\psi}{r^2} \right)_r \sin^2 (u)
%%			+\frac{\alpha^2}{2} \int_{0}^\infty  \left(\frac{\psi}{r^4}\right)_r  \sin^4 (u) \\
%					=&
%			-\frac12 \int_{0}^\infty  \psi' \left(  1+\frac{2 \alpha^2 \sin^2(u)}{r^2} \right) (u_t^2+u_r^2) 
%			+2\alpha^2\int_{0}^\infty \frac{\psi}{r} \frac{\sin^2(u)}{r^2} (u_t^2+u_r^2)\\
%			&
%			+2\int_{0}^\infty  \frac{\psi}{r} u_r^2
%			+\int_{0}^\infty  \left(\frac{\psi'}{r^2}-2\frac{\psi}{r^3} \right) \sin^2 (u)
%			+\frac{\alpha^2}{2} \int_{0}^\infty  \left(\frac{\psi'}{r^2}-4\frac{\psi}{r^3}\right)  \frac{\sin^4 (u)}{r^2} \\
				=&
			-\frac12 \int_{0}^\infty  \psi' \left(  1+\frac{2 \alpha^2 \sin^2(u)}{r^2} \right) (u_t^2+u_r^2) 
			+2\alpha^2\int_{0}^\infty \frac{\psi}{r} \frac{\sin^2(u)}{r^2} (u_t^2+u_r^2)\\
			&
			+2\int_{0}^\infty  \frac{\psi}{r} u_r^2
			+\int_{0}^\infty  \frac{\psi'}{r^2}\sin^2 (u) \left(1
			+\frac{\alpha^2}{2}   \frac{\sin^2 (u)}{r^2}\right) \\
			&-2\int_{0}^\infty \frac{\psi}{r^3}  \sin^2 (u)\left( 1+\alpha^2  \frac{\sin^2 (u)}{r^2}\right)\\
			=&-\frac12 \int_{0}^\infty  \frac{\psi'}{r^2}  r^2 u_t^2 
			-\frac12 \int_{0}^\infty  p(r) r^2 u_r^2 
			-\frac{\alpha^2}{2} \int_{0}^\infty p(r)  \sin^2(u) (u_t^2+u_r^2)\\
			&
			+\int_{0}^\infty  \left(\frac{\psi'}{r^2}-2\frac{\psi}{r^3}\right)\sin^2 (u) 
			+\frac{\alpha^2}{2}  \int_{0}^\infty  p(r)   \frac{\sin^4 (u)}{r^2},
% 			\footnote{No veo pq esta expresion es la \eqref{eq:dt_KS}....creo que falta un paso final para llegar a  \eqref{eq:dt_KS}}
	%		
	\end{aligned}
	\]
	and we conclude.
	\end{proof}

Similarly, we have the following result for the correction term $\Ps$ \eqref{P}.

\begin{lem}[]\label{lem:P'}
	Let $t\in \R$, $\phi$ be a smooth weight function and $\Ps$ as in \eqref{P}. Then, if $u\in \mathcal{E}_0^{S,r\phi}$, we have
	\begin{equation}\label{eq:dt_PS}
	\begin{aligned}
		%\dfrac{d}{dt}\Ps=&~{}
			\frac{d}{dt} \Ps
		=&~{}
	\alpha^2 \int_{0}^\infty \frac{ \phi}{r^2} \left( \sin (2u)  u  + 2 \sin^2 (u) \right) (u_t^2-u_r^2) + \int_{0}^\infty  \phi  (u_{t}^2-u_r^2)\\
		&
	% &-\alpha^2\int_{0}^\infty  \frac{\phi}{r^2} \left( \sin(2u) u +2  \sin^2(u) \right) u_{r}^2 \\
	-\int_{0}^\infty   \left[
					 \left( \frac{\phi}{r} \right)_r  -\frac12   \phi'' 
					\right] u^2
	-\int_{0}^\infty  \phi  \dfrac{\sin (2u)}{r^2} u\\
	&+\alpha^2\int_{0}^\infty \left( r\phi'' -4\phi' +6\frac{\phi}{r}\right)  \frac{\sin^2(u)}{r^3}  u^2 -
		\alpha^2 \int_{0}^\infty  \dfrac{\phi}{r}  \dfrac{u \sin (2u) \sin^2(u)}{r^3} \\
		&+ \alpha^2\int_{0}^\infty \left(\phi'-2 \frac{\phi}{r} \right)   \frac{\sin(2u)}{r^2} u_r u^2 .
%	\alpha^2 \int_{0}^\infty \frac{ \phi}{r^2} \left( \sin (2u)  u  + 2 \sin^2(u) \right) (u_t^2-u_r^2) \\
%		&+\alpha^2\int_{0}^\infty \left( r\phi'' -4\phi' +6\frac{\phi}{r}\right)  \frac{\sin^2(u)}{r^3}  u^2 -\alpha^2 \int_{0}^\infty  \dfrac{\phi}{r}  \dfrac{u \sin (2u) \sin^2(u)}{r^3} \\
%		&+ \alpha^2\int_{0}^\infty \left(\phi'-2 \frac{\phi}{r} \right)   \frac{\sin(2u)}{r^2} u_r u^2 \\
%	% &-\alpha^2\int_{0}^\infty  \frac{\phi}{r^2} \left( \sin(2u) u +2  \sin^2(u) \right) u_{r}^2 \\
%	&-\int_{0}^\infty    \left( \frac{\phi}{r} \right)_r u^2 
%	+\frac12 \int_{0}^\infty  \phi''  u^2
%	-\int_{0}^\infty  \phi  \dfrac{ \sin(2u)}{r^2} u
%	-\int_{0}^\infty  \phi  u_{r}^2+ \int_{0}^\infty  \phi  u_{t}^2 .
	\end{aligned}
	\end{equation}
\end{lem}
\begin{proof}
	Derivating the functional \eqref{P} with respect to time, we have
	\begin{equation}\label{eq:dt_PS_initial}
	\begin{aligned}
	\frac{d}{dt} \Ps
	=&\int_{0}^\infty  \phi \left( f(u) \right)_t u_t u 
	+\int_{0}^\infty  \phi f(u) (u_{tt} u+u_t^2) \\
	=&~{}2\alpha^2 \int_{0}^\infty  \phi \frac{\sin (2u)}{r^2}  u_t^2 u ~{} +\int_{0}^\infty  \phi f(u) u_t^2 
	+\int_{0}^\infty  \phi  f(u) u_{tt} u \\
	:=&P_1+P_2+P_3.
	\end{aligned}
	\end{equation}
	Using \eqref{skyrme}  in $P_3$, we get
	\[
	\begin{aligned}
	P_3
	=&\int_{0}^\infty  \phi  u \left\{
	f(u) u_{rr}+\frac{2}{r} u_r-\dfrac{ \sin(2u)}{r^2} \left[1+\alpha^2\left(u^2_t -u_r^2 +\dfrac{\sin^2(u)}{r^2} \right)
	\right] \right\}.\\
	\end{aligned}
	\]
	Integrating by parts the first term on the RHS, we get
	\begin{equation*}
	\begin{aligned}
	\int_{0}^\infty  \phi   f(u) u u_{rr}
%	=& 	-\int_{0}^\infty  (\phi  f(u) u)_r u_{r}\\
%	=& 	-\int_{0}^\infty  (\phi'   f(u) u+\phi (f(u))_r u+\phi  f(u) u_r) u_{r}\\
%	=& 	-\int_{0}^\infty  (\phi'   f(u) u u_{r} +\phi (f(u))_r u u_{r}+\phi  f(u) u_r^2) \\
	=& 	-\int_{0}^\infty  \phi'   f(u) u u_{r} -\int_{0}^\infty  \phi (f(u))_r u u_{r} -\int_{0}^\infty  \phi  f(u) u_r^2 \\
	=&~{} \frac12 \int_{0}^\infty  (\phi''   f(u)+\phi'   (f(u))_r ) u^2  -\int_{0}^\infty  \phi (f(u))_r u u_{r} -\int_{0}^\infty  \phi  f(u) u_r^2. 
	\end{aligned}
	\end{equation*}
	Having in mind derivatives in \eqref{f_der}, we get
	\begin{equation}\label{eq:P3b}
	\begin{aligned}
	\int_{0}^\infty  \phi   f(u) u u_{rr}
%	=&~{} \frac12 \int_{0}^\infty  \phi''   f(u)u^2
%	+ \frac12\int_{0}^\infty \phi'   \left(2\alpha^2\frac{\sin(2u)u_r}{r^2}-4\alpha^2\frac{\sin^2(u)}{r^3} \right) u^2 \\
%	& -\int_{0}^\infty  \phi \left( 2\alpha^2\frac{\sin(2u)u_r}{r^2}-4\alpha^2\frac{\sin^2(u)}{r^3}\right) u u_{r} \\
%	& -\int_{0}^\infty  \phi  f(u) u_r^2 \\
	=&~{} \frac12 \int_{0}^\infty  \phi''   f(u)u^2
	+ \alpha^2\int_{0}^\infty \phi'   \left(\frac{\sin(2u)u_r}{r^2}-2\frac{\sin^2(u)}{r^3} \right) u^2 \\
	& -2\alpha^2\int_{0}^\infty  \phi \left( \frac{\sin(2u)u_r}{r^2}-2\frac{\sin^2(u)}{r^3}\right) u u_{r} 
	 -\int_{0}^\infty  \phi  f(u) u_r^2 \\
	=&~{} \frac12 \int_{0}^\infty  \phi''   f(u)u^2
	+ \alpha^2\int_{0}^\infty \phi'   \frac{\sin(2u)}{r^2} u_r u^2 -2\alpha^2\int_{0}^\infty \phi'  \frac{\sin^2(u)}{r^3}  u^2 \\
	& -2\alpha^2\int_{0}^\infty  \phi \frac{\sin(2u)}{r^2} u u_{r}^2 +4\alpha^2\int_{0}^\infty  \phi \frac{\sin^2(u)}{r^3} u u_{r} 
	 -\int_{0}^\infty  \phi  f(u) u_r^2 .
	\end{aligned}
	\end{equation}
	Now, integrating by parts the second term  in the RHS of the above line, we obtain
	\begin{equation}\label{eq:P3_penult}
	\begin{aligned}
	4\alpha^2\int_{0}^\infty  \phi \frac{\sin^2(u)}{r^3} u u_{r}
	=& -2\alpha^2\int_{0}^\infty \left (\phi \frac{\sin^2(u)}{r^3}\right)_r u^2\\
	=& -2\alpha^2 \int_{0}^\infty  \phi' \frac{\sin^2(u)}{r^3}u^2 \\
	&-2\alpha^2\int_{0}^\infty  \phi\left( \frac{\sin(2u)u_r}{r^3}-3\frac{\sin^2(u)}{r^4}\right)u^2.
	\end{aligned}
	\end{equation}
	 Then, substituting into 
	 %replacing on $P_3$ the term obtained when \eqref{eq:P3_penult} is used on  
	 \eqref{eq:P3b},
	 %, and regrouping,
	 we get
	\[\begin{aligned}
	P_3
	 =&~{} 
	\frac12 \int_{0}^\infty  \phi''   f(u)u^2 
	+\alpha^2\int_{0}^\infty \left(-4\phi' +6\frac{\phi}{r}\right)  \frac{\sin^2(u)}{r^3}  u^2 -\alpha^2 \int_{0}^\infty  \dfrac{\phi}{r}  \dfrac{u \sin (2u) \sin^2(u)}{r^3} \\
		&+ \alpha^2\int_{0}^\infty \left(\phi'-2 \frac{\phi}{r} \right)   \frac{\sin(2u)}{r^2} u_r u^2 
	 -\int_{0}^\infty  \phi \left(\alpha^2\frac{\sin(2u)}{r^2} u + f(u)\right) u_{r}^2 \\
	%& -\int_{0}^\infty  \phi  f(u) u_r^2  \\
	&-\int_{0}^\infty    \left( \frac{\phi}{r} \right)_r u^2 
	-\int_{0}^\infty  \phi  \dfrac{ \sin(2u)}{r^2} u
	-\alpha^2 \int_{0}^\infty  \phi  \dfrac{ \sin(2u)}{r^2} u u^2_t  .
	\end{aligned}
	\]
	Replacing \eqref{f} and regrouping again, we obtain
	\begin{equation}\label{eq:P3_final}\begin{aligned}
	P_3
%	 =&~{} 
%	\frac12 \int_{0}^\infty  \phi''   f(u)u^2 \\
%	&+\alpha^2\int_{0}^\infty \left(-4\phi' +6\frac{\phi}{r}\right)  \frac{\sin^2(u)}{r^3}  u^2 -\alpha^2 \int_{0}^\infty  \dfrac{\phi}{r}  \dfrac{u \sin (2u) \sin^2(u)}{r^3} \\
%		&+ \alpha^2\int_{0}^\infty \left(\phi'-2 \frac{\phi}{r} \right)   \frac{\sin(2u)}{r^2} u_r u^2 \\
%	 &-\int_{0}^\infty  \phi \left(\alpha^2\frac{\sin(2u)}{r^2} u + f(u)\right) u_{r}^2 \\
%	%& -\int_{0}^\infty  \phi  f(u) u_r^2  \\
%	&-\int_{0}^\infty     \frac{\phi}{r} u^2 
%	-\int_{0}^\infty  \phi  \dfrac{ \sin(2u)}{r^2} u\\
%	&-\alpha^2 \int_{0}^\infty  \phi  \dfrac{ \sin(2u)}{r^2} u u^2_t  
%	 \\
	  =&~{} 
	%\alpha^2 \int_{0}^\infty  r\phi''  \frac{\sin^2(u)}{r^3} u^2 \\
	\alpha^2\int_{0}^\infty \left( r\phi'' -4\phi' +6\frac{\phi}{r}\right)  \frac{\sin^2(u)}{r^3}  u^2 -\alpha^2 \int_{0}^\infty  \dfrac{\phi}{r}  \dfrac{u \sin (2u) \sin^2(u)}{r^3} \\
		&+ \alpha^2\int_{0}^\infty \left(\phi'-2 \frac{\phi}{r} \right)   \frac{\sin(2u)}{r^2} u_r u^2 
	-\alpha^2\int_{0}^\infty  \frac{\phi}{r^2} \left( \sin(2u) u +2  \sin^2(u) \right) u_{r}^2 \\
	&+ \int_{0}^\infty     \left(\frac12 \phi''-\left( \frac{\phi}{r} \right)_r \right) u^2 
	%+\frac12 \int_{0}^\infty  \phi''  u^2
	-\int_{0}^\infty  \phi  \dfrac{ \sin(2u)}{r^2} u
	-\int_{0}^\infty  \phi  u_{r}^2 
	-\alpha^2 \int_{0}^\infty  \phi  \dfrac{ \sin(2u)}{r^2} u u^2_t . 
	\end{aligned}
	\end{equation}
	Collecting $P_3$ in \eqref{eq:P3_final}, \eqref{eq:dt_PS_initial}, and using \eqref{f}, we get
	\[
	\begin{aligned}
	\frac{d}{dt} \Ps
	=&~{}
	2\alpha^2 \int_{0}^\infty  \phi \frac{\sin (2u)}{r^2}  u_t^2 u +\int_{0}^\infty  \phi f(u) u_t^2 
		+\alpha^2\int_{0}^\infty \left( r\phi'' -4\phi' +6\frac{\phi}{r}\right)  \frac{\sin^2(u)}{r^3}  u^2\\
		& -\alpha^2 \int_{0}^\infty  \dfrac{\phi}{r}  \dfrac{u \sin (2u) \sin^2 (u)}{r^3} \\
		&+ \alpha^2\int_{0}^\infty \left(\phi'-2 \frac{\phi}{r} \right)   \frac{\sin(2u)}{r^2} u_r u^2 
	 -\alpha^2\int_{0}^\infty  \frac{\phi}{r^2} \left( \sin(2u) u +2  \sin^2 (u) \right) u_{r}^2 \\
	&-\int_{0}^\infty    \left( \frac{\phi}{r} \right)_r u^2 
	+\frac12 \int_{0}^\infty  \phi''  u^2
	-\int_{0}^\infty  \phi  \dfrac{\sin (2u)}{r^2} u
	-\int_{0}^\infty  \phi  u_{r}^2 
	-\alpha^2 \int_{0}^\infty  \phi  \dfrac{\sin (2u)}{r^2} u u^2_t.  \\
	%%
	%%
%		=&~{}
%	\alpha^2 \int_{0}^\infty \frac{ \phi}{r^2} \left( \sin (2u)  u  + 2 \sin^2(u) \right) (u_t^2-u_r^2)\\
%		&+\alpha^2\int_{0}^\infty \left( r\phi'' -4\phi' +6\frac{\phi}{r}\right)  \frac{\sin^2(u)}{r^3}  u^2 -\alpha^2 \int_{0}^\infty  \dfrac{\phi}{r}  \dfrac{u \sin (2u) \sin^2(u)}{r^3} \\
%		&+ \alpha^2\int_{0}^\infty \left(\phi'-2 \frac{\phi}{r} \right)   \frac{\sin(2u)}{r^2} u_r u^2 \\
%	% &-\alpha^2\int_{0}^\infty  \frac{\phi}{r^2} \left( \sin(2u) u +2  \sin^2(u) \right) u_{r}^2 \\
%	&-\int_{0}^\infty    \left( \frac{\phi}{r} \right)_r u^2 
%	+\frac12 \int_{0}^\infty  \phi''  u^2
%	-\int_{0}^\infty  \phi  \dfrac{ \sin(2u)}{r^2} u
%	-\int_{0}^\infty  \phi  u_{r}^2+ \int_{0}^\infty  \phi  u_{t}^2 \\
	\end{aligned}
	\]
	Finally, regrouping  terms, we conclude
		\[
	\begin{aligned}
	\frac{d}{dt} \Ps
		=&~{}
	\alpha^2 \int_{0}^\infty \frac{ \phi}{r^2} \left( \sin (2u)  u  + 2 \sin^2 (u) \right) (u_t^2-u_r^2) + \int_{0}^\infty  \phi  (u_{t}^2-u_r^2)\\
		&+\alpha^2\int_{0}^\infty \left( r\phi'' -4\phi' +6\frac{\phi}{r}\right)  \frac{\sin^2(u)}{r^3}  u^2 -
		\alpha^2 \int_{0}^\infty  \dfrac{\phi}{r}  \dfrac{u \sin (2u) \sin^2(u)}{r^3} \\
		&+ \alpha^2\int_{0}^\infty \left(\phi'-2 \frac{\phi}{r} \right)   \frac{\sin(2u)}{r^2} u_r u^2 
	% &-\alpha^2\int_{0}^\infty  \frac{\phi}{r^2} \left( \sin(2u) u +2  \sin^2(u) \right) u_{r}^2 \\
	-\int_{0}^\infty   \left[
					 \left( \frac{\phi}{r} \right)_r  -\frac12   \phi'' 
					\right] u^2
	-\int_{0}^\infty  \phi  \dfrac{\sin (2u)}{r^2} u.
	\end{aligned}
	\]
	%This concludes the proof of this lemma.
	% \[
	% \begin{aligned}
	% \left(\phi  \left(r^n+r^{n-2}2 \alpha^2 \sin^2(u) \right) \right)_{rr}
	% =& \phi_{rr}  \left(r^n+r^{n-2}2 \alpha^2 \sin^2(u) \right)
	% +2\phi_r  \left(r^n+r^{n-2}2 \alpha^2 \sin^2(u) \right)_r \\
	% &+\phi  \left(r^n+r^{n-2}2 \alpha^2 \sin^2(u) \right)_{rr}\\
	% 	=& \phi_{rr}  \left(r^n+r^{n-2}2 \alpha^2 \sin^2(u) \right)\\
	% &+2\phi_r  \left(nr^{n-1}+(n-2)r^{n-3}2 \alpha^2 \sin^2(u) +r^{n-2}2 \alpha^2 \sin(2u)u_r \right) \\
	% &+\phi  \left[n(n-1)r^{n-2}
	% +(n-2)(n-3)r^{n-4}2 \alpha^2 \sin^2(u)\right.\\
	% &+\left. 2(n-2)r^{n-3}2 \alpha^2 \sin(2 u)u_r+r^{n-2}2 \alpha^2 (\sin (2u)u_{rr}+2\cos(2u)u^2_r) \right]
	% \end{aligned}
	% \]
\end{proof}

\subsection{Virial identities for the Adkins-Nappi Model}\label{Sect:virial_AdkinsNappi}

		\medskip

%		Let
%	%	$\lambda(t), \mu(t)$ a function that has never zero time scaling be a large parameter and
%		$\rho=\rho(t,r)$ a smooth, weighted function, to be chosen later. For each $t\in \R$ we consider the following functional
Let $\rho=\rho(t,r)$ a smooth, weight function, to be chosen later. Similarly to the previous section, for the Adkins-Nappi equation we introduce a suitable functional, 
 as a weighted generalization of the energy  \eqref{energy_AN}, and given by 
		\begin{equation}\label{J}
		\begin{aligned}
%
	%	\mathcal J_1(t;\lambda)= \JJ =& \int_{0}^{\infty} \rho\left(\dfrac{r}{\lambda(t)} \right) r^2 \left[ u_t^2+u^2_r+ 2\dfrac{\sin^2(u)}{r^2}+\dfrac{\left(u- \sin (u) \cos (u) \right)^2 }{r^4} \right] \label{J}, \\
		%\mathcal{I}_{AN}(t;\sigma,L)=
		 \JJ =& \int_{0}^{\infty} \rho r^2 \left[ u_t^2+u^2_r
		+ 2\dfrac{\sin^2(u)}{r^2}+\dfrac{\left(u- \sin (u) \cos (u) \right)^2 }{r^4} \right],\ \mbox{for each}\  t\in\R .
			%\hbox{and} &   \nonumber\\
		%	\mathcal J(t;L,\sigma)= \JJ  =& \Int{\R} \rho\left(\dfrac{x+\sigma t}{L}\right)\left( uv+ \partial_x u  \partial_x v\right) (t,x)dx \label{J}.
		\end{aligned}
		\end{equation}
		Recalling Remark \ref{rem:well_defined_spaces}, if $\rho$ is a bounded function, the functional $\JJ$
		 is well-defined for $(u,u_t)\in (\dot{H}^{5/3} \cap\dot{H}^1 \times L^2) (\R^3) $. 
		 The following result describes the time variation of \eqref{J}.
		%This fact is essential for the proof, and it is the key ingredient in both Theorems \ref{thm:AN_exterior}-\ref{thm:AN_interior}.

		\begin{lem}[Energy local variations: Adkins-Nappi Model]\label{lem:dtJ}
			For any $t\in \R$, one has
			\begin{align}\label{eq:dtJ}
				\dfrac{d}{dt}\JJ =&~{} \int_{0}^{\infty} \rho_t r^2 \left[ u_t^2+u^2_r+ 2\dfrac{\sin^2(u)}{r^2}
				+\dfrac{\left(u- \sin (u) \cos (u) \right)^2 }{r^4} \right] 
				- 2\int_{0}^{\infty} \rho_r r^2 u_t u_r.
			\end{align}
		\end{lem}

		\begin{proof}Derivating the functional \eqref{J} with respect to
		time and using basic trigonometric identities, we obtain
			\begin{equation}\label{J'_start}
			\begin{aligned}
	 \dfrac{d}{dt}\JJ
	 % =& \int_{0}^{\infty} \rho_t r^2 \left[ u_t^2+u^2_r+ 2\dfrac{\sin^2(u)}{r^2}+\dfrac{\left(u- \sin (u) \cos (u) \right)^2 }{r^4} \right] +2\int_{0}^{\infty} \rho r^2 u_r u_{rt}\\
	 % &+ \int_{0}^{\infty} \rho r^2 \left[ 2u_t u_{tt}+ 4\dfrac{u_t \sin (u) \cos (u)}{r^2}+\dfrac{2u_t \left(u- \sin (u) \cos (u) \right) \left(1+  \sin^2(u) - \cos^2 u \right)}{r^4} \right]\\
	 =& \int_{0}^{\infty} \rho_t r^2 \left[ u_t^2+u^2_r+ 2\dfrac{\sin^2(u)}{r^2}+\dfrac{\left(u- \sin (u) \cos (u) \right)^2 }{r^4} \right]  +2\int_{0}^{\infty} \rho r^2 u_r u_{rt}\\
	 &+ \underbrace{\int_{0}^{\infty} 2 \rho u_t  r^2 \left[ u_{tt}+ \dfrac{ \sin (2u)}{r^2}+\dfrac{ \left(u- \sin (u) \cos (u) \right) \left(1+  \sin^2(u) - \cos^2 u \right)}{r^4} \right] }_{J_1}.
			\end{aligned}
		\end{equation}
Now, using the equation \eqref{AdkinsNappi} and integrating by parts  in $J_1$, we have
% u_{tt}=+u_{rr}+\dfrac{2}{r}u_r-\dfrac{ \sin(2u)}{r^2}- \dfrac{\left( u- \sin (u) \cos (u)\right) \left(1-\cos (2u) \right)}{r^4}=0

\[
\begin{aligned}
J_1
=& 2\int_{0}^{\infty}  \rho u_t  r^2
\left[ u_{rr}+\dfrac{2}{r}u_r \right] \\
% =& \int_{0}^{\infty} 4 \rho  r u_t  u_r  
% + \int_{0}^{\infty} 2 \rho   r^2  u_t u_{rr} \\
% =& \int_{0}^{\infty} 4 \rho  r u_t  u_r  + (2\rho   r^2  u_tu_r)\bigg|_{r=0}^{r=\infty}\\
% &- \int_{0}^{\infty} 2 \bigg(\rho   r^2  u_t\bigg)_{r} u_{r} \\
=& 4 \int_{0}^{\infty}  \rho  r u_t  u_r  
%+ {\color{red}(2\rho   r^2  u_tu_r)\bigg|_{r=0}^{r=\infty}}
- 2\int_{0}^{\infty}  \bigg(\rho_r   r^2  u_t+\rho   2r  u_t+\rho   r^2  u_{tr}\bigg) u_{r}  .\\
\end{aligned}
\]
Finally, substituting $J_1$ in \eqref{J'_start}, we get
\begin{equation}\label{J'_ends}
\begin{aligned}
\dfrac{d}{dt}\JJ
% =& \int_{0}^{\infty} \rho_t r^2 \left[ u_t^2+u^2_r+ 2\dfrac{\sin^2(u)}{r^2}+\dfrac{\left(u- \sin (u) \cos (u) \right)^2 }{r^4} \right] +2\int_{0}^{\infty} \rho r^2 u_r u_{rt}\\
% &+ \int_{0}^{\infty} \rho r^2 \left[ 2u_t u_{tt}+ 4\dfrac{u_t \sin (u) \cos (u)}{r^2}+\dfrac{2u_t \left(u- \sin (u) \cos (u) \right) \left(1+  \sin^2(u) - \cos^2 u \right)}{r^4} \right]\\
=& \int_{0}^{\infty} \rho_t r^2 \left[ u_t^2+u^2_r+ 2\dfrac{\sin^2(u)}{r^2}+\dfrac{\left(u- \sin (u) \cos (u) \right)^2 }{r^4} \right]  \\
&-2\int_{0}^{\infty}  \rho_r  r^2 u_t u_{r}.
\end{aligned}
\end{equation}
This ends the proof of the lemma.
		\end{proof}

\begin{rem}\label{rem:J'}
Similarly to Remark \ref{rem:I'}, using the change of variables $\rho=\phi/r^2,$ the term $r^2$ in $\JJ$ is avoided,  and therefore recovering the functional $E_{AN,\phi}[u]$  \eqref{w_energy_AN}. Furthermore, by Lemma \ref{lem:dtJ} we have the following identity for 
the time variation of $E_{AN,\phi}$:
% to avoid the $r^2$ weight given by the dimensionality of the problem, it is enough to consider $\rho=\phi/r^2$. Then, we get the following relation
\begin{equation}\label{J'_wo_r2}
\begin{aligned}
\dfrac{d}{dt}E_{AN,\phi}(t)
=& \int_{0}^{\infty} \phi_t  \left[ u_t^2+u^2_r+ 2\dfrac{\sin^2(u)}{r^2}+\dfrac{\left(u- \sin (u) \cos (u) \right)^2 }{r^4} \right] \\
&-2\int_{0}^{\infty}  \left(\phi_r-2\frac{\phi}{r}\right)  u_t u_{r}.
\end{aligned}
\end{equation}
This relation will be useful in the proof of Theorem \ref{thm:AN_interior}.
\end{rem}

					\noindent

		%Now, we define a \emph{weighted version of a sort of momentum}, as follows:

		Now, let $\psi$ and $\phi$ smooth weight functions of $r$, which will be chosen later. We define the functional 
 $\Man$ associated with a sort of momentum, given by
\begin{equation}\label{momentum}
		\Man =\int_{0}^{\infty} \psi u_t u_r.
		\end{equation}
\noindent 
and  the functional $\Ran$, which is the term  that corrects the bad sign of the variation on the functional $\Man$, given by
	\begin{equation}\label{R}
	\Ran
	=\int_{0}^{\infty} \phi u_t u.
	\end{equation}
	\noindent
	The following results show the time variation of these functionals, which  will be used 
	in the proof of Theorem \ref{thm:AN_interior}.
%.In the following result is given an identity for the time variation of the functional above-defined:
		\begin{lem}
		Let $t\in \R$, $\psi$ be a smooth weight function and $\Man$ as in \eqref{momentum}. Then, if $u\in \mathcal{E}_0^{AN,\psi}$, we have
			%For any $t\in \R$, one has
			\begin{equation}\label{eq:dt_M}
			\begin{aligned}
			\dfrac{d}{dt} \Man 
			=&
			%\int_{0}^{\infty} \psi_t  u_t u_r 
			-\frac12 \int_{0}^{\infty} \psi'  u_t^2 
			- \int_{0}^{\infty}  \left(\frac{\psi'}{2} -\dfrac{2\psi }{r}\right)u_r^2
			-\frac12 \int_{0}^{\infty} \left( 2 \frac{\psi}{r} -\psi'\right)\frac{\sin^2(u) }{r^2} \\
			&-\frac12\int_{0}^{\infty} \left(  4 \frac{\psi}{r} -\psi' \right) \frac{( u- \sin (u) \cos (u))^2}{r^4}.
			\end{aligned}
			\end{equation}
		\end{lem}
		\begin{proof}
			Just derivating the functional \eqref{momentum} with respect to time,
			%and using the identities $\sin(2u)=\sin (u) \cos (u)$ and $\cos(2u)=\cos^2 u -\sin^2(u)$
			we obtain
			\[
			\begin{aligned}
			\dfrac{d}{dt} \Man
			=&% \int_{0}^{\infty} \psi_t u_t u_r +
			\int_{0}^{\infty} \psi (u_{tt} u_r+u_t u_{rt}) \\
			=& %\int_{0}^{\infty} \psi_t  u_t u_r 
			-\frac12 \int_{0}^{\infty} \psi'  u_t^2 
			%+{\color{red}\frac12 \psi u_t^2\bigg|_{r=0}^{r=\infty}}
			+\int_{0}^{\infty} \psi  u_{tt} u_r 
			:=M_1+M_2.
			%\\
			%=& \int_{0}^{\infty} \varphi_t r^2 u_t u_r 
			%+\int_{0}^{\infty} \varphi r^2 (u_{tt} u_r+\dfrac{(u_t^2)_r}{2} )\\
			\end{aligned}
			\]
			For $M_2$, using \eqref{AdkinsNappi} and integrating by parts, we have
			\[
			\begin{aligned}
			M_2
			=& \int_{0}^{\infty} \psi  u_r \left(u_{rr}+\dfrac{2}{r}u_r\right)-\int_{0}^{\infty} \psi u_r \left(\dfrac{ \sin(2u)}{r^2}+
			\dfrac{\left( u- \sin (u) \cos (u)\right) \left(1-\cos (2u) \right)}{r^4} \right)\\
			=&- \int_{0}^{\infty}  \left(\frac{\psi'}{2} -\dfrac{2\psi }{r}\right)u_r^2
			-\int_{0}^{\infty} \psi  u_r \left(\dfrac{ \sin(2u)}{r^2}+
			\dfrac{\left( u- \sin (u) \cos (u)\right) \left(1-\cos (2u) \right)}{r^4} \right)\\
		%	&+{\color{red}\frac12\psi  u_{r}^2\bigg|_{r=0}^{r=\infty}}\\
			=&{}~ M_{21}+M_{22}.
			\end{aligned}
			\]
			With respect to $M_{22}$, note that rewriting it and integrating by parts, we obtain
			%Focus on $M_{32}$. Rewriting and integrating by parts, we obtain			
			\[
			\begin{aligned}
			M_{22}
			%=&\frac12 \int_{0}^{\infty} \varphi  (\cos (2u))_r
			%-\frac12\int_{0}^{\infty} \varphi
			%\dfrac{\left(( u- \sin (u) \cos (u))^2\right)_r }{r^2} \right)\\
			=&
			-\frac12 \int_{0}^{\infty} \frac{\psi}{r^2}  (\sin^2(u))_r
			-\frac12\int_{0}^{\infty} \frac{\psi}{r^4}
			\left(( u- \sin (u) \cos (u))^2\right)_r  \\
%			=&
%			-\frac12 \int_{0}^{\infty} \psi\left[ \frac{(\sin^2(u))_r r^2-2r \sin^2(u)  }{r^4} +\frac{2 \sin^2(u) }{r^3} \right] \\
%			&-\frac12\int_{0}^{\infty} \psi \left[ \frac{\left(( u- \sin (u) \cos (u))^2\right)_r r^4- 4( u- \sin (u) \cos (u))^2 r^3}{r^8}+4\frac{( u- \sin (u) \cos (u))^2}{r^5}	\right]		  \\
%			=&
%			-\frac12 \int_{0}^{\infty} \psi \left[ \frac{\sin^2(u)}{r^2} \right]_r- \int_{0}^{\infty} \psi \frac{\sin^2(u) }{r^3} \\
%			&-\frac12\int_{0}^{\infty} \psi \left[ \frac{( u- \sin (u) \cos (u))^2 }{r^4}\right]_r - 2\int_{0}^{\infty} \psi \frac{( u- \sin (u) \cos (u))^2}{r^5}		  \\
%			=&
%			\frac12 \int_{0}^{\infty} \psi_r \frac{\sin^2(u)}{r^2} - \int_{0}^{\infty} \frac{\psi}{r} \frac{\sin^2(u) }{r^2} \\
%			&+\frac12\int_{0}^{\infty} \psi_r  \frac{( u- \sin (u) \cos (u))^2 }{r^4} - 2\int_{0}^{\infty} \frac{\psi}{r} \frac{( u- \sin (u) \cos (u))^2}{r^4}		  \\
				=&
			-\frac12 \int_{0}^{\infty} \left( 2 \frac{\psi}{r} -\psi' \right)\frac{\sin^2(u) }{r^2} 
			-\frac12\int_{0}^{\infty} \left(  4 \frac{\psi}{r} -\psi'  \right) \frac{( u- \sin (u) \cos (u))^2}{r^4}.
			%
			%%
%			=&
%			\frac12 \int_{0}^{\infty} \left(\frac{\psi}{r^2} \right)_r \sin^2(u)
%			-{\color{red}\left(\frac{\psi}{r^{4}}\right)
%				( u- \sin (u) \cos (u))^2-\left(\frac{\psi}{r^{2}}\right)
%				-\sin^2(u)\bigg|_{r=0}^{r=\infty}}+\frac12\int_{0}^{\infty} \left(\frac{\psi}{r^{4}}\right)_r
%			( u- \sin (u) \cos (u))^2  \\
%			=&
%			\frac12 \int_{0}^{\infty} \left(\frac{\psi}{r^2} \right)_r \sin^2(u)
%			+\frac12\int_{0}^{\infty} \left(\frac{\psi}{r^{4}}\right)_r
%			\left( u- \frac12 \sin (2u)\right)^2  \\
			%
			%	&-\int_{0}^{\infty} (\varphi_r r^2+2\varphi r) \dfrac{u_t^2}{2}  
			\end{aligned}
			\]
			Finally, collecting $M_1$, $M_{21}$,  and $M_{22}$,  we get
			\[
			\begin{aligned}
			\dfrac{d}{dt}\Man 
			=&
			%\int_{0}^{\infty} \psi_t  u_t u_r 
			-\frac12 \int_{0}^{\infty} \psi'  u_t^2 
			- \int_{0}^{\infty}  \left(\frac{\psi'}{2} -\dfrac{2\psi }{r}\right)u_r^2\\
			&-\frac12 \int_{0}^{\infty} \left( 2 \frac{\psi}{r} -\psi' \right)\frac{\sin^2(u) }{r^2} 
			-\frac12\int_{0}^{\infty} \left(  4 \frac{\psi}{r} -\psi' \right) \frac{( u- \sin (u) \cos (u))^2}{r^4}.
%			=& \int_{0}^{\infty} \psi_t  u_t u_r 
%			-\frac12 \int_{0}^{\infty} \psi_r  u_t^2 
%			- \int_{0}^{\infty}  \left(\frac{\psi_r}{2} -\dfrac{2\psi }{r}\right)u_r^2\\
%			&+\frac12 \int_{0}^{\infty} \left(\frac{\psi}{r^2} \right)_r \sin^2(u)
%			+\frac12\int_{0}^{\infty} \left(\frac{\psi}{r^{4}}\right)_r
%			\left( u- \frac12 \sin (2u)\right)^2 
			%\\
			%=& \int_{0}^{\infty} \varphi_t r^2 u_t u_r 
			%+\int_{0}^{\infty} \varphi r^2 (u_{tt} u_r+\dfrac{(u_t^2)_r}{2} )\\
			\end{aligned}
			\]
			This ends the proof of this lemma.
		\end{proof}
%		Now, we define a correction term as follows:
%	\begin{equation}\label{R}
%	\Ran
%	=\int_{0}^{\infty} \phi u_t u.
%	\end{equation}
%	\noindent
%	The following results shows the time variation of correction:
	
	\begin{lem}
	Let $t\in \R$, $\phi$ be a smooth weight function and $\Ran$ as in \eqref{R}. Then, if $u\in \mathcal{E}_0^{AN,r \phi}$, we have
		%For any $t\in \R$, one has
		\begin{equation}\label{eq:dt_R}
		\begin{aligned}
		\dfrac{d}{dt} \Ran 
		=& %\int_{0}^{\infty} \phi_t  u_t u +
		 \int_{0}^{\infty} \phi  u_t^2 
-\int_{0}^{\infty}  \left[  \phi' r-\phi  -\frac{r^2\phi_{rr}}{2}\right]\frac{u^2}{r^2}
-\int_{0}^{\infty} \phi u_r^2\\
&-\int_{0}^{\infty}  \left(\dfrac{\phi}{r^2}\right) u \sin(2u)
-\int_{0}^{\infty} \left(
\dfrac{\phi}{r^4}\right) u \left( u- \sin (u) \cos (u) \right) (1-\cos (2u)).
		\end{aligned}
		\end{equation}
	\end{lem}
	\begin{proof}
		Just derivating the functional \eqref{R} with respect to time,
		%and using the identities $\sin(2u)=\sin (u) \cos (u)$ and $\cos(2u)=\cos^2 u -\sin^2(u)$
		we obtain
		\begin{equation}\label{eq:R1+R2+R3}
		\begin{aligned}
		\dfrac{d}{dt}\Ran 
		%=& \int_{0}^{\infty} \phi_t u_t u 
		%+\int_{0}^{\infty} \phi (u_{tt} u+ u_t^2)\\
		=&% \int_{0}^{\infty} \phi_t  u_t u +
		 \int_{0}^{\infty} \phi  u_t^2 
		+\int_{0}^{\infty} \phi  u_{tt} u 
		:=R_1+R_2.
		%\\
		%=& \int_{0}^{\infty} \varphi_t r^2 u_t u_r 
		%+\int_{0}^{\infty} \varphi r^2 (u_{tt} u_r+\dfrac{(u_t^2)_r}{2} )\\
		\end{aligned}
		\end{equation}
		For $R_2$, using \eqref{AdkinsNappi} and integrating by parts, we get
		\[
		\begin{aligned}
		R_2
		%=& -\int_{0}^{\infty} (\phi' u +\phi u_r)  u_{r}- \int_{0}^{\infty}  \left(\dfrac{\phi}{r}\right)_r u^2-\int_{0}^{\infty} \phi u \left(\dfrac{ \sin(2u)}{r^2}+ \dfrac{\left( u- \sin (u) \cos (u) \right) (1-\cos (2u)) }{r^4} \right)\\
		=& \int_{0}^{\infty} \frac{\phi_{rr}}{2} u^2 -\int_{0}^{\infty} \phi u_r^2 
		- \int_{0}^{\infty}  \left(\dfrac{\phi}{r}\right)_r u^2\\
		&-\int_{0}^{\infty} \phi u \left(\dfrac{ \sin(2u)}{r^2}+
		\dfrac{\left( u- \sin (u) \cos (u) \right) (1-\cos (2u)) }{r^4} \right).
		\end{aligned}
		\]
		Regrouping  terms, we obtain
		\begin{equation}\label{eq:R3}
		\begin{aligned}
	R_2	=& \int_{0}^{\infty}  \left(\frac{\phi_{rr}r^2}{2} -r \phi_r+\phi  \right)\frac{u^2}{r^2}-\int_{0}^{\infty} \phi u_r^2\\
		&-\int_{0}^{\infty}  \left(\dfrac{\phi}{r^2} u \sin(2u)+
			\dfrac{\phi}{r^4}\left( u- \sin (u) \cos (u) \right) (1-\cos (2u)) u \right).\\
		%
	%	&-\int_{0}^{\infty} \psi  u_r \left(\dfrac{ \sin(2u)}{r^2}+
	%	\dfrac{\left( u- \sin (u) \cos (u)\right) \left(1-\cos (2u) \right)}{r^4} \right)\\
	%	=& M_{21}+M_{22}
		\end{aligned}
		\end{equation}
		Then, substituting \eqref{eq:R3} in \eqref{eq:R1+R2+R3}, we obtain
		\[
		\begin{aligned}
		\dfrac{d}{dt}\Ran
		=& %\int_{0}^{\infty} \phi_t  u_t u +
		 \int_{0}^{\infty} \phi  (u_t^2-u_r^2) 
-\int_{0}^{\infty}  \left[  \phi' r-\phi  -\frac{r^2\phi_{rr}}{2}\right]\frac{u^2}{r^2}\\
&
-\int_{0}^{\infty}  \dfrac{\phi}{r^2} u \sin(2u)
-\int_{0}^{\infty} 
\dfrac{\phi}{r^4} u \left( u- \sin (u) \cos (u) \right) (1-\cos (2u)).
		%\\
		%=& \int_{0}^{\infty} \varphi_t r^2 u_t u_r 
		%+\int_{0}^{\infty} \varphi r^2 (u_{tt} u_r+\dfrac{(u_t^2)_r}{2} )\\
		\end{aligned}
		\]
		This concludes the proof of the lemma.
	\end{proof}
%{\color{red}
%	\begin{rem}
%	One can see the similarities in the virials identities for the Skyrme and Adkins-Nappi equations. In both cases, we have to consider a correction term that introduces quadratic terms of $ u $, allowing simplify the problem taking off the nonlinear functions involved. Furthermore, {\color{red} aasdasd}
%	\end{rem}
%	}
	
%%%	%%
%%%	%% Here
%%%	%%
		\section{Decay in  exterior light cones for the Skyrme and Adkins-Nappi models}\label{Sect:4}

This section deals with the proof of Theorem \ref{thm:S_exterior} for the Skyrme and Adkins-Nappi equations. %First, we prove decay in exterior light cones.
%
%		\subsection{Proof of Theorem \ref{thm:S_exterior} (Decay in exterior light cones for the Skyrme and Adkins-Nappi equations)} In this Section we prove Theorem \ref{thm:S_exterior}. 
		In what follows, fix $\sigma\in \R$ such that $|\sigma|>1$. Recalling the identity \eqref{derI1} and using the weight function $\varphi=\varphi\left(\dfrac{r+\sigma t}{L}\right)$, we get
			\begin{equation}
			\begin{aligned}
				\dfrac{d}{dt}\II=&~{}\dfrac{\sigma}{L}\int_{0}^{\infty} \varphi' r^2 \left[ \left(1+\dfrac{2 \alpha^2 \sin^2(u)}{r^2} \right) (u_t^2+u^2_r)+ 2\dfrac{\sin^2(u)}{r^2}+\dfrac{\alpha^2 \sin^4(u)}{r^4} \right] \\
				&-\frac{1}{L}\int_{0}^{\infty} \varphi' r^2  \left(1+\dfrac{2 \alpha^2 \sin^2(u)}{r^2} \right) 2 u_t u_r .\label{u_t u_r}
			\end{aligned}
		\end{equation}
		Furthermore, from Lemma \ref{lem:dtJ}, we have:
\begin{equation}
\begin{aligned}
\dfrac{d}{dt}\JJ =&~{} \dfrac{\sigma}{L} \int_{0}^{\infty} \varphi' r^2 \left[ u_t^2+u^2_r+ 2\dfrac{\sin^2(u)}{r^2}+\dfrac{\left(u- \sin (u) \cos (u) \right)^2 }{r^4} \right] \\
&- \dfrac{1}{L}\int_{0}^{\infty} \varphi' 2r^2 u_t u_r .
\label{u_t u_r AdkinsNappi}
\end{aligned}
\end{equation}
			 Now, we are ready to prove a first virial estimate.\footnote{We denote by $a \lesssim_{h} b$ if there is a constant $C$ depending on $h$ such that $a\leq C(h) b$. }

			\begin{lem}
			%For $r>0$, assume  \varphi\left(\dfrac{r +\sigma t}{L} \right)
			 Let $L>0$, $\sigma=-(1+b)<-1$, and $\rho=\varphi=\tanh\left(\dfrac{r+\sigma t}{L}\right)$. Then  
			 \begin{enumerate}
				\item 
				\begin{equation}
					\frac{d}{dt}\II\lesssim_{L,b} -\int_{0}^{\infty} \varphi' r^2 \left[
					\left(1+\dfrac{2 \alpha^2 \sin^2(u)}{r^2} \right) (u_t^2+u^2_r)+ 2\dfrac{\sin^2(u)}{r^2}+\dfrac{\alpha^2 \sin^4(u)}{r^4} \right].
				 \label{sim_norm_S}
				\end{equation}
			\item 
			\begin{equation}
			\frac{d}{dt}\JJ \lesssim_{L,b} -\int_{0}^{\infty} \rho' r^2 \left[ u_t^2+u^2_r+ 2\dfrac{\sin^2(u)}{r^2}+\dfrac{\left(u- \sin (u) \cos (u) \right)^2 }{r^4} \right].
				\label{sim_norm}
				\end{equation}
			\end{enumerate}
			\end{lem}

			\begin{proof}
			 Firstly we prove \eqref{sim_norm_S}. Focusing on the last term in the RHS of  \eqref{u_t u_r},
			note that, if $\varphi'> 0$, then using a Cauchy-Schwarz inequality, we have
			\begin{align*}
				\left|\int_{0}^{\infty} \varphi' r^2  \left(1+\dfrac{2 \alpha^2 \sin^2(u)}{r^2} \right) 2 u_t u_r \right|
				%	\leq \left(\Int{\R} \varphi'v^2\right)^{1/2}\left(\Int{\R} \varphi'f^2\right)^{1/2}
				\leq \int_{0}^{\infty} \varphi' r^2  \left(1+\dfrac{2 \alpha^2 \sin^2(u)}{r^2} \right) (u_t^2+ u_r^2).
			\end{align*}
			Therefore, if $b>0$, $\sigma=-(1+b)<-1$, and $\varphi=\tanh$, we have from \eqref{u_t u_r}
			\[
				\begin{aligned}
					\dfrac{d}{dt}\II \leq &~{}
					\dfrac{\sigma}{L}\int_{0}^{\infty} \varphi' r^2 \left[ \left(1+\dfrac{2 \alpha^2 \sin^2(u)}{r^2} \right) (u_t^2+u^2_r)+ 2\dfrac{\sin^2(u)}{r^2}+\dfrac{\alpha^2 \sin^4(u)}{r^4} \right] \\
					&+\frac{1}{L}\int_{0}^{\infty} \varphi' r^2  \left(1+\dfrac{2 \alpha^2 \sin^2(u)}{r^2} \right) (u_t^2+ u_r^2) .
%%					=&~{} \dfrac{\sigma}{L}\int_{0}^{\infty} \varphi' r^2 \left[ 2\dfrac{\sin^2(u)}{r^2}+\dfrac{\alpha^2 \sin^4(u)}{r^4} \right]\\
%%					&+\frac{1+\sigma}{L}\int_{0}^{\infty} \varphi' r^2  \left(1+\dfrac{2 \alpha^2 \sin^2(u)}{r^2} \right) (u_t^2+ u_r^2)\\
%					=&~{} -\frac{b}{L}\int_{0}^{\infty} \varphi' r^2  \left(1+\dfrac{2 \alpha^2 \sin^2(u)}{r^2} \right) (u_t^2+ u_r^2)\\
%					& -\dfrac{1+b}{L}\int_{0}^{\infty} \varphi' r^2 \left[ 2\dfrac{\sin^2(u)}{r^2}+\dfrac{\alpha^2 \sin^4(u)}{r^4} \right]\\		
%					\leq &~{} -\frac{b}{L}\int_{0}^{\infty} \varphi' r^2  
%											 \left[ 
%												\left(1+\dfrac{2 \alpha^2 \sin^2(u)}{r^2} \right) (u_t^2+ u_r^2)
%												+2\dfrac{\sin^2(u)}{r^2}+\dfrac{\alpha^2 \sin^4(u)}{r^4}
%											 \right] \\
				%	& -\dfrac{b}{L}\int_{0}^{\infty} \varphi' r^2 \left[ 2\dfrac{\sin^2(u)}{r^2}+\dfrac{\alpha^2 \sin^4(u)}{r^4} \right] 				%	
					%%
					%% MEJORAR?  SE PUEDE OBTENER MAS INFORMACION?
					%%
				%	\leq& -\frac{b}{L}\int_{0}^{\infty} \varphi' r^2 (u_t^2+ u_r^2) 
%
				\end{aligned}
				\]
			Consequently, we obtain \eqref{sim_norm_S}
				\begin{equation*}\label{intermedio_S}
				\begin{aligned}
					\dfrac{d}{dt}\II
%					\leq &~{}
%					\dfrac{\sigma}{L}\int_{0}^{\infty} \varphi' r^2 \left[ \left(1+\dfrac{2 \alpha^2 \sin^2(u)}{r^2} \right) (u_t^2+u^2_r)+ 2\dfrac{\sin^2(u)}{r^2}+\dfrac{\alpha^2 \sin^4(u)}{r^4} \right]\\
%					&+\frac{1}{L}\int_{0}^{\infty} \varphi' r^2  \left(1+\dfrac{2 \alpha^2 \sin^2(u)}{r^2} \right) (u_t^2+ u_r^2)\\
%%					=&~{} \dfrac{\sigma}{L}\int_{0}^{\infty} \varphi' r^2 \left[ 2\dfrac{\sin^2(u)}{r^2}+\dfrac{\alpha^2 \sin^4(u)}{r^4} \right]\\
%%					&+\frac{1+\sigma}{L}\int_{0}^{\infty} \varphi' r^2  \left(1+\dfrac{2 \alpha^2 \sin^2(u)}{r^2} \right) (u_t^2+ u_r^2)\\
%					=&~{} -\frac{b}{L}\int_{0}^{\infty} \varphi' r^2  \left(1+\dfrac{2 \alpha^2 \sin^2(u)}{r^2} \right) (u_t^2+ u_r^2)\\
%					& -\dfrac{1+b}{L}\int_{0}^{\infty} \varphi' r^2 \left[ 2\dfrac{\sin^2(u)}{r^2}+\dfrac{\alpha^2 \sin^4(u)}{r^4} \right]\\		
					\lesssim_{L,b} &~{} -\int_{0}^{\infty} |\varphi'| r^2  
											 \left[ 
												\left(1+\dfrac{2 \alpha^2 \sin^2(u)}{r^2} \right) (u_t^2+ u_r^2)
												+2\dfrac{\sin^2(u)}{r^2}+\dfrac{\alpha^2 \sin^4(u)}{r^4}
											 \right]  .
				\end{aligned}
				\end{equation*}
%		Integrating in time, we have proved \eqref{integrability0_S_ext} in Theorem \ref{thm:S_exterior}.

\medskip 

%Similarly for the Adkins-Nappi equation, we concentrate on the last term in \eqref{u_t u_r AdkinsNappi}.  $\varphi'> 0$, using a Cauchy-Schwarz inequality, we have
The proof of \eqref{sim_norm} proceeds in a similar way. Only note that the last term in \eqref{u_t u_r AdkinsNappi} verifies the following inequality
\begin{align*}
\left|\int_{0}^{\infty} \varphi' r^2  2 u_t u_r \right|
%	\leq \left(\Int{\R} \varphi'v^2\right)^{1/2}\left(\Int{\R} \varphi'f^2\right)^{1/2}
\leq \int_{0}^{\infty} \varphi' r^2 (u_t^2+ u_r^2), \mbox{ for } \varphi'> 0,
\end{align*}
the rest of the proof follows the same lines as in the Skyrme case and hence, for the sake of 
simplicity, we do not show it here.\\
%\noindent
%Consequently, if $b>0$, $\sigma=-(1+b)<-1$, and $\varphi=\tanh$, we have from \eqref{u_t u_r AdkinsNappi}
%\[
%\begin{aligned}
%\dfrac{d}{dt}\JJ
%\leq&~ \dfrac{\sigma}{L} \int_{0}^{\infty} \varphi' r^2 \left[ u_t^2+u^2_r+ 2\dfrac{\sin^2(u)}{r^2}+\dfrac{\left(u- \sin (u) \cos (u) \right)^2 }{r^4} \right] 
%+ \dfrac{1}{L}\int_{0}^{\infty} \varphi' r^2 (u_t^2+ u_r^2) \\
%=&~ \dfrac{\sigma}{L} \int_{0}^{\infty} \varphi' r^2 \left[ 2\dfrac{\sin^2(u)}{r^2}+\dfrac{\left(u- \sin (u) \cos (u) \right)^2 }{r^4} \right] 
%+ \dfrac{\sigma+1}{L}\int_{0}^{\infty} \varphi' r^2 (u_t^2+ u_r^2) .\\
%%
%\end{aligned}
%\]
%\noindent
%Finally we obtain
%\begin{equation}\label{intermedio_AN}
%\frac{d}{dt}\JJ \lesssim_{L,b}   -\int_0^{\infty} |\varphi'|r^2 \left( u_t^2+u_r^2+2\dfrac{\sin^2(u)}{r^2}+\dfrac{\left(u- \sin (u) \cos (u) \right)^2 }{r^4}\right)dr,
%\end{equation}
%\noindent
%and integrating in time, we have proved \eqref{integrability0_S_ext} in Theorem \ref{thm:AN_exterior}.

Finally, we can observe that integrating in time on \eqref{sim_norm_S} and \eqref{sim_norm}, we have proved \eqref{integrability0_S_ext} in Theorem \ref{thm:S_exterior}.
\end{proof}

		%Integrating in time, we have proved \eqref{integrability0_S_ext} in Theorem \ref{thm:S_exterior}.

\subsection{Proof of Theorem \ref{thm:S_exterior}:  Skyrme and Adkins-Nappi equations} 
Firstly, we focus on the Skyrme case.  It only remains to prove \eqref{decay0_S_ext}.
We must to show decay in the right hand side region, namely $((1+b)t,+\infty)$, $b>0$.
Now we choose $\varphi(r)=\frac{1}{2}\left(1+\tanh(r)\right)$, $\sigma=-(1+b)$, and  $\tilde{\sigma}=-(1+b/2)$ with $b>0$. Consider the modified energy functional, for $t\in[2,t_0]$:

\begin{align*}
&\mathcal{I}_{S,t_0}(t):= \\
&\dfrac{1}{2} \int_{0}^{\infty} \varphi\left(\dfrac{r+\sigma t_0-\tilde{\sigma}(t_0-t)}{L}\right) r^2 \left[ \left(1+\dfrac{2 \alpha^2 \sin^2(u)}{r^2} \right) (u_t^2+u^2_r)+ 2\dfrac{\sin^2(u)}{r^2}+\dfrac{\alpha^2 \sin^4(u)}{r^4} \right].
\end{align*}

\medskip
\noindent
Note that $\sigma<\tilde{\sigma}<0$. From Lemma \ref{lem:I'} and proceeding exactly as in \eqref{sim_norm_S}, we have
{\small
\[
\begin{aligned}
\frac{d }{dt}\mathcal{I}_{S,t_0}(t) \lesssim_{b,L}& \\
 -\int_{0}^{\infty}  \mbox{sech}^2 &\left(\dfrac{r+\sigma t_0-\tilde{\sigma}(t_0-t)}{L}\right)r^2
\left[ 
\left(1+\dfrac{2 \alpha^2 \sin^2(u)}{r^2} \right) (u_t^2+u^2_r)+ 2\dfrac{\sin^2(u)}{r^2}+\dfrac{\alpha^2 \sin^4(u)}{r^4}
\right]\\
 \leq& 0,
\end{aligned}
\]
}
\medskip
\noindent
which means that the new functional $\mathcal{I}_{S,t_0}$ is decreasing in $[2,t_0]$. Therefore, we have
\begin{align*}
\int_{2}^{t_0} \dfrac{d}{dt} \mathcal{I}_{S,t_0}(t) dt= \mathcal{I}_{S,t_0}(t_0)-\mathcal{I}_{S,t_0}(2)\leq 0\implies
\ \mathcal{I}_{S,t_0}(t_0)\leq\mathcal{I}_{S,t_0}(2).
\end{align*}
\noindent
On the other hand, since $\lim_{x\to -\infty} \varphi (x)=0,$  we have
{\small
\begin{align*}
\limsup_{t\to\infty} \int_{0}^{\infty} \varphi\left(\dfrac{r-\beta t-\gamma}{L}\right)r^2\left[ 
\left(1+\dfrac{2 \alpha^2 \sin^2(u)}{r^2} \right) (u_t^2+u^2_r)+ 2\dfrac{\sin^2(u)}{r^2}+\dfrac{\alpha^2 \sin^4(u)}{r^4}
\right](\nu,r)=0,\\
\end{align*}}
\noindent
for $\beta,\gamma,\nu>0$ fixed. This yields
{\small
\[
\begin{aligned}
0 & \leq \int_{0}^{\infty} \varphi\left(\dfrac{r-(1+b)t_0}{L}\right)r^2\left[ \left(1+\dfrac{2 \alpha^2 \sin^2(u)}{r^2} \right) (u_t^2+u^2_r)+ 2\dfrac{\sin^2(u)}{r^2}+\dfrac{\alpha^2 \sin^4(u)}{r^4} \right](t_0,r)\\
& \leq  \int_{0}^{\infty} \varphi\left(\dfrac{r- \frac b2t_0-(2+b)}{L}\right)r^2\left[\left(1+\dfrac{2 \alpha^2 \sin^2(u)}{r^2} \right) (u_t^2+u^2_r)+ 2\dfrac{\sin^2(u)}{r^2}+\dfrac{\alpha^2 \sin^4(u)}{r^4}\right](2,r),
\end{aligned}
\]
}
\noindent
which implies, 
\[
\begin{aligned}
\limsup_{t\to\infty}\int_{0}^{\infty} \varphi\left(\dfrac{r-(1+b)t}{L}\right)r^2
\bigg[ 
\left(1+\dfrac{2 \alpha^2 \sin^2(u)}{r^2} \right) (u_t^2+u^2_r)
\quad \quad \quad \quad \quad \quad \quad&\\+ 2\dfrac{\sin^2(u)}{r^2}+\dfrac{\alpha^2 \sin^4(u)}{r^4}
\bigg](t,r)dr&=0.\\
\end{aligned}
\]
This means that the energy over $R(t)$ (see \eqref{S}) converges to zero, implying \eqref{decay0_S_ext} and then we conclude the Skyrme case. 
%(EN P.15, HAY TDV EXPRESIONES QUE SE SALEN DEL MARGEN; CORREGIR ESTO)

\medskip
\noindent
For the Adkins-Nappi case, the proof is analogous but this time considering the modified energy functional
\begin{align*}
\mathcal{I}_{AN,t_0}(t):=  \int_{0}^{\infty} \rho\left(\dfrac{r+\sigma t_0-\tilde{\sigma}(t_0-t)}{L}\right) r^2
\left[ u_t^2+u^2_r+ 2\dfrac{\sin^2(u)}{r^2}+\dfrac{\left(u- \sin (u) \cos (u) \right)^2 }{r^4} \right],
\end{align*}
and repeating the same steps as in the Skyrme case. This concludes the proof of Theorem \ref{thm:AN_exterior}.

			\medskip
\section{Decay of weighted energies}\label{Sect:5}

Firstly, we study the growth rate of the modified energies introduced in \eqref{w_energy_S} and \eqref{w_energy_AN}.

					\subsection{Growth rate for the modified energy in the Skyrme and Adkins-Nappi equations}
	
	In this section we study the growth rate for the power type weighted energy of the Skyrme and Adkins-Nappi equations
	
	\begin{prop}\label{prop:time_rate_skyrme}
	Let $u$ a global solution of \eqref{skyrme} (or \eqref{AdkinsNappi}) such that $u\in \overset{n}{\underset{i=2}{\bigcap}} \mathcal{E}_{0}^{X,r^i}$, 
	for $X=S$ or $X=AN$. Then the corresponding weighted energy satisfies
	\begin{equation*}
	E_{X,r^n}[u](t)=O(t^{n-2}),
	\end{equation*}
	where $E_{X,r^n}[u](t)$ is given in  \eqref{w_energy_S} and  \eqref{w_energy_AN}, respectively .
	\end{prop}	
	
%		\begin{prop}
%	Let $u$ be a global solution of \eqref{AdkinsNappi} such that $u\in \overset{n}{\underset{i=2}{\bigcap}}  \mathcal{E}_{0}^{AN,r^i}$. Then the energy satisfies
%	\begin{equation}
%	E_{AN,r^n}[u](t)=O(t^{n-2})
%	\end{equation}
%	where $E_{AN,r^n}[u](t)$ is given in  \eqref{w_energy_AN}.
%	\end{prop}	
	
	\begin{proof}
	Firstly, we consider $X=S$. We note that for $\varphi=\phi/r^2$, we get
	\begin{equation}\label{eq:J=ES}
	\II=E_{S,\phi}[u](t).
	%\int_{0}^{\infty} \varphi(r) \left[ u_t^2+u^2_r+ 2\dfrac{\sin^2(u)}{r^2}+\dfrac{\left(u- \sin (u) \cos (u) \right)^2 }{r^4} \right] 
	\end{equation}
	Then, using \eqref{J'_wo_r2} with $\phi=r^n$, one can see 
	\[
	\dfrac{d}{dt}\II=-2\Ks,
	\]
	where $\Ks$ is given by \eqref{K} and $\psi=\phi'-2\frac{\phi}{r}=(n-2)r^{n-1}$. Now using \eqref{eq:J=ES}, we get\\
%	Then, using that $\I=E_{S,r^n}$ one can see 
	\begin{equation*}
		\left|\dfrac{d}{dt}E_{S,r^n}[u] (t)\right|\lesssim E_{S,r^{n-1}}[u](t),
	\end{equation*}
	and for $n=3$, we obtain
	\begin{equation*}
	\begin{aligned}
		\left|\dfrac{d}{dt}E_{S,r^3}[u](t)\right|\lesssim& E_{S,r^{2}}[u](t)=E_{S}[u](0),\\
		\left|E_{S,r^3}[u](t)\right|\lesssim& E_{S}[u](0)t+|E_{S,r^3}[u](0)|.
	\end{aligned}
	\end{equation*}
	Similarly, for $n=4$ and using the last inequality, we get
	\begin{equation*}
	\begin{aligned}
		\left|\dfrac{d}{dt}E_{S,r^4} [u](t)\right|\lesssim& E_{S,r^{3}}[u](t) \lesssim E_{S}[u](0) t +|E_{S,r^3}[u](0)|.
	\end{aligned}
	\end{equation*}
	Now, integrating with respect of time, we have
	\begin{equation*}
	\begin{aligned}
		\left|E_{S,r^4}[u](t)\right|\lesssim& E_{S}[u](0) \frac{t^2}{2} +|E_{S,r^3}[u](0)| t +|E_{S,r^4}[u](0)|.
	\end{aligned}
	\end{equation*}
	Repeating this procedure, we conclude
		\begin{equation*}
	\begin{aligned}
		\left|E_{S,r^n}[u](t)\right|\lesssim& E_{S}[u](0) t^{n-2} +\sum_{j=0}^{n-3} t^{j} |E_{S,r^{n-j}}[u](0)|.
	\end{aligned}
	\end{equation*}
	This ends the proof for the case $X=S$. 
	Analogously, following the same ideas, it can be proved for the case $X=AN$ case.	This completes the proof.
	\end{proof}	
%
%{\color{red}		
%		As a direct consequence, we have the following corollary (YO NO PONDRIA ESTE COROLARIO PUES ES 
%		DIRECTO DE LA PROP ANTERIOR)
%	\begin{cor}
%	\begin{enumerate}
%	\item Let $u$ a global solution of \eqref{skyrme} such that $u\in \overset{6}{\underset{i=2}{\bigcap}} \mathcal{E}_{0}^{S,r^i}$. Then, the weighted energy holds
%	\begin{equation*}
%	E_{S,r^6}[u](t)=O(t^4),
%	\end{equation*}
%	where $E_{S,r^6}[u](t)$ is given in  \eqref{w_energy_S}.
%
%	\item Let $u$ be a global solution of \eqref{AdkinsNappi} such that $u\in \overset{4}{\underset{i=2}{\bigcap}}\mathcal{E}_{0}^{AN,r^i}$. Then the energy satisfies
%	\begin{equation*}
%	E_{AN,r^4}[u](t)=O(t^2)
%	\end{equation*}
%	where $E_{AN,r^4}[u](t)$ is given in  \eqref{w_energy_AN}.
%	\end{enumerate}
%	\end{cor}	
%}	

\subsection{Decay to zero for modified Energies: Proof of the Theorem \ref{thm:S_interior}}

In the spirit of \cite{ACKM,AM,MM}, we consider a suitable linear combination of virials  $\Ks$ and $\Ps$ (see \eqref{K} and \eqref{P}), 
and $\Man$ and $\Ran$ (see \eqref{momentum} and \eqref{R}), for the Skyrme and Adkins-Nappi models respectively. Let 
\begin{equation}\label{eq:HS}
\Hs=\Ks+\gamma_S \Ps,
		\end{equation} 	
		and
		\begin{equation}\label{eq:Han}\Han =\Man+\gamma_{AN} \Ran,
\end{equation}
\\
\noindent
be new virials, where $\gamma_S$ and $\gamma_{AN}$ will be chosen later. These new virials introduce $u^2$ terms, which allow us to simplify the problem considering  
Taylor expansions for the involved trigonometric functions.

%Similarly to the Skyrme equation, we consider a suitable linear combination of the virial $\Man$ and $\Ran$, defined in \eqref{momentum} and \eqref{R}, respectively. Leting 
%\begin{equation}\label{eq:Han}\Han =\Man+\gamma \Ran,
%\end{equation}
%\noindent
% where $\gamma $ will be chosen later. As in the Skyrme equation, $\Han$ introduces $u^2$ terms, and with suitable Taylor expansion, allows us to simplify the problem

   \subsubsection{Decay to zero for modified Energy: Proof of the Theorem \ref{thm:S_interior} for the Skyrme model}

\begin{lem}\label{lem:HS_psi}
Let $u$ be a global solution of \eqref{skyrme} such that $\| u\|_{L^{\infty}}\leq \delta$,  $u\in \mathcal{E}_{0}^{S,\psi}$, and 
$\psi=r \phi$ 
(where $\psi$ and $\phi$ are the weight functions presented in \eqref{K} and \eqref{P}).
%Let $\psi=r^5/(1+r)$, $\phi=r^4/(1+r)$ and $\gamma_S=1/6$. 
Then, $\Hs$  in \eqref{eq:HS} satisfies the following identity
{\small
\begin{equation*}
\begin{aligned}
\frac{d}{dt}\Hs =
			& 	 -\frac12 \int_{0}^\infty  \left( \psi' -2 \gamma_S \frac{\psi}{r} \right)  u_t^2 
			-\frac12 \int_{0}^\infty  \left( \psi' +(2 \gamma_S  -4) \frac{\psi}{r}\right) u_r^2  \\
					&	 -\int_{0}^\infty    \left( (2-\gamma_S) \frac{\psi}{r} +(2\gamma_S-1)  \psi' - \frac{\gamma_S}{2} r \psi'' \right) \frac{u^2}{r^2} 
						  - \alpha^2 \int_{0}^\infty  
							 \left( \psi'  -(2+4\gamma_S) \frac{\psi}{r}   \right) \frac{u^2}{r^2}  (u_t^2+u_r^2)  \\					
	& -\alpha^2    \int_{0}^\infty  
						\left(-\frac12(1-6\gamma_S) \psi'  +(2-4\gamma_S)  \frac{\psi}{r}  -\frac{\gamma_S}{2}  r \psi''  \right)  \frac{u^4}{r^4} 
+H_e(t),
\end{aligned}
\end{equation*}
}
where {\small
\begin{equation*}
\begin{aligned}
H_e(t) =&\frac19\alpha^2\int_{0}^\infty \left( 
							- \gamma_S r\psi''
							+\left( -3   +6 \gamma_S  \right) \psi'
							 +\left( 6 \gamma_S  +  12  \right) \frac{\psi}{r}
	 					\right) \left(\frac{u^6}{r^4}+\frac{O(u^8)}{r^4}\right)\\	
						&-\frac13 \int_{0}^\infty  \left( \psi'  -2(1+2\gamma_S) \frac{\psi}{r}\right) \frac{u^4}{r^2}
			+\frac2{45}\int_{0}^\infty  \left( \psi'  -2\left( 1+\gamma_S \right) \frac{\psi}{r}
			 \right) \left(\frac{u^6}{r^2}+\frac{O(u^8)}{r^2}\right) \\
			 &- \alpha^2 \int_{0}^\infty  
						\bigg[
							 \frac13\left(- \psi'  +2 (1+6\gamma_S) \frac{\psi}{r}  \right) u^2\\
							&\quad \quad \quad \quad \quad \quad \quad  +\frac2{45}\left( \psi'  -2(1+4\gamma_S) \frac{\psi}{r} \right) \left( u^4+O(u^6) \right) 
						\bigg] \frac{u^2}{r^2}(u_t^2+u_r^2).
\end{aligned}
\end{equation*}
}

 \end{lem}

\begin{proof}
Collecting \eqref{eq:dt_PS} and \eqref{eq:dt_KS} and regrouping terms, we get
{\small
\[
\begin{aligned}
\frac{d}{dt}\mathcal{H}_{S}(t)=
%b=\gamma_S
%	=&
%			-\frac12 \int_{0}^\infty  \psi'  u_t^2 
%			+b \int_{0}^\infty  \phi  u_{t}^2
%			-\frac12 \int_{0}^\infty  \psi'  u_r^2 
%			-b\int_{0}^\infty  \phi  u_{r}^2 
%			+2\int_{0}^\infty  \frac{\psi}{r} u_r^2\\
%	&+\int_{0}^\infty  \frac{\psi'}{r^2}\sin^2 (u)  	
%			-2\int_{0}^\infty \frac{\psi}{r^3}  \sin^2 (u)
%			-b\int_{0}^\infty     \frac{\phi}{r} u^2 
%			+b\frac12 \int_{0}^\infty  \phi''  u^2
%			-b\int_{0}^\infty  \phi  \dfrac{ \sin(2u)}{r^2} u \\
%	&-\frac12 \int_{0}^\infty  \psi' \frac{2 \alpha^2 \sin^2(u)}{r^2}  u_t^2 
%			+2\alpha^2\int_{0}^\infty \frac{\psi}{r} \frac{\sin^2(u)}{r^2} u_t^2 
%			+b\alpha^2 \int_{0}^\infty \frac{ \phi}{r^2} \left( \sin (2u)  u  + 2 \sin^2(u) \right) u_t^2\\
%	&-\frac12 \int_{0}^\infty  \psi' \frac{2 \alpha^2 \sin^2(u)}{r^2}  u_r^2 
%			+2\alpha^2\int_{0}^\infty \frac{\psi}{r} \frac{\sin^2(u)}{r^2} u_r^2 
%			-b\alpha^2 \int_{0}^\infty \frac{ \phi}{r^2} \left( \sin (2u)  u  + 2 \sin^2(u) \right) u_r^2\\
%	&+b\alpha^2\int_{0}^\infty \left( r\phi'' -4\phi' +6\frac{\phi}{r}\right)  \frac{\sin^2(u)}{r^3}  u^2 
%			-b\alpha^2 \int_{0}^\infty  \dfrac{\phi}{r}  \dfrac{u \sin (2u) \sin^2(u)}{r^3} 
%			 +\int_{0}^\infty  \frac{\psi'}{r^2} \frac{\alpha^2}{2}   \frac{\sin^4 (u)}{r^2}
%			-2\int_{0}^\infty \frac{\psi}{r^3}  \alpha^2  \frac{\sin^4 (u)}{r^2} \\
%	&+b \alpha^2\int_{0}^\infty \left(\phi'-2 \frac{\phi}{r} \right)   \frac{\sin(2u)}{r^2} u_r u^2 \\
		&
			-\frac12 \int_{0}^\infty  \left( \psi' -2\gamma_S \phi \right)  u_t^2 
			-\frac12 \int_{0}^\infty  \left( \psi' +2\gamma_S  \phi  -4 \frac{\psi}{r}\right) u_r^2 \\
	&+\int_{0}^\infty  \left( \frac{\psi'}{r^2}  -2 \frac{\psi}{r^3}\right)\sin^2 (u)  	
			-\gamma_S\int_{0}^\infty   \dfrac{ \phi}{r^2} u  \sin(2u)
			- \gamma_S \int_{0}^\infty    \left( \left( \frac{\phi}{r} \right)_r - \frac12 \phi''  \right)u^2  \\
				&+\gamma_S\alpha^2\int_{0}^\infty \left( \frac{\phi''}{r^2} -4\frac{\phi'}{r^3} +6\frac{\phi}{r^4}\right)  \sin^2(u)  u^2 
			- \gamma_S\alpha^2 \int_{0}^\infty  \dfrac{\phi}{r^4}  u \sin (2u) \sin^2(u) \\
	&+\frac{\alpha^2}{2}    \int_{0}^\infty  \left( \frac{\psi'}{r^4}  -4  \frac{\psi}{r^5}    \right)\sin^4 (u)
	+\gamma_S \alpha^2\int_{0}^\infty \left(\frac{\phi'}{r^2}-2 \frac{\phi}{r^3} \right)   \sin(2u) u_r u^2 \\
	&- \alpha^2 \int_{0}^\infty  
						\left[
							 \frac{\psi'}{r^2}    \sin^2(u)  
							-2 \frac{\psi}{r^3} \sin^2(u)
							- \gamma_S  \frac{ \phi}{r^2} \left( \sin (2u)  u  + 2 \sin^2(u) \right) 
						\right] (u_t^2+u_r^2) .
\end{aligned}
\]}
Now, let $\delta>0$ small enough such that $\| u\|_{L^\infty}<\delta$ (by Remark \ref{rem:energy_Linfty}), we note
\begin{equation}\label{eq:taylor_S}
\begin{aligned}
\sin^2(u)=& u^2-\frac13 u^4+\frac2{45} u^6+O(u^8), \\
u\sin(2u)=& 2u^2-\frac43 u^4+\frac4{15} u^6+O(u^8), \\
2\sin^2(u)+u\sin(2u)=& 4u^2-2u^4+\frac{16}{45}u^6+O(u^8), \\
u \sin^2(u)\sin(2u)=& 2u^4-2u^6+O(u^8),\\%+\frac45 u^8+O(u^{10}) \\
\sin^4(u)=& u^4-\frac23 u^6+O(u^8).%\frac15 u^8+O(u^{10})
\end{aligned}
\end{equation}
Then, we obtain the following decomposition
\begin{equation*}\label{DtH}
\frac{d}{dt} \Hs = H_1+H_2+H_3+H_4+H_5,\end{equation*} 
\noindent
where 
\[
\begin{aligned}
		H_1=&
			-\frac12 \int_{0}^\infty  \left( \psi' -2 \gamma_S\phi \right)  u_t^2 
			-\frac12 \int_{0}^\infty  \left( \psi' +2 \gamma_S \phi  -4 \frac{\psi}{r}\right) u_r^2 , \\
	H_2=&{}~ \int_{0}^\infty  \left( \frac{\psi'}{r^2}  -2 \frac{\psi}{r^3}\right)\left[ u^2-\frac13 u^4+\frac2{45} u^6+O(u^8)\right]  	
			- \gamma_S\int_{0}^\infty    \left(  \left( \frac{\phi}{r} \right)_r - \frac12 \phi''  \right) u^2  \\
			&- \gamma_S \int_{0}^\infty   \dfrac{ \phi}{r^2} \left(  2u^2-\frac43 u^4+\frac4{15} u^6+O(u^8) \right),\\
	H_3=&- \alpha^2 \int_{0}^\infty  
						\bigg[
							 \left(\frac{\psi'}{r^2}  -2 \frac{\psi}{r^3} \right) \left( u^2-\frac13 u^4+\frac2{45} u^6+O(u^8)\right)\\
							&\quad \quad \quad \quad\quad \quad \quad \quad-  \gamma_S  \frac{ \phi}{r^2} \left( 4u^2-2u^4+\frac{16}{45}u^6+O(u^8) \right) 
						\bigg] (u_t^2+u_r^2),\\
%	H_4=&- \alpha^2 \int_{0}^\infty  
%						\left[
%							\left(\frac{\psi'}{r^2}   - 2\frac{\psi}{r^3} \right)\left( u^2-\frac13 u^4+\frac2{45} u^6+O(u^8)\right)
%							+b  \frac{ \phi}{r^2} \left( 4u^2-2u^4+\frac{16}{45}u^6+O(u^8) \right) 
%							\right] u_r^2,\\
%%	H_4=&{}~  \gamma_S\alpha^2\int_{0}^\infty \left( \frac{\phi''}{r^2} -4\frac{\phi'}{r^3} +6\frac{\phi}{r^4}\right)  \left( u^4-\frac13 u^6+O(u^{8})\right)
%%			- \gamma_S \alpha^2 \int_{0}^\infty  \dfrac{\phi}{r^4}  \left(  2u^4-2u^6+O(u^{8}) \right) \\
%%	&+\frac{\alpha^2}{2}    \int_{0}^\infty  \left( \frac{\psi'}{r^4}  -4  \frac{\psi}{r^5}    \right) \left(  u^4-\frac23 u^6+O(u^{8}) \right),\\
	H_4=&{}~  \gamma_S\alpha^2\int_{0}^\infty \bigg[ \bigg( \frac{\phi''}{r^2} -4\frac{\phi'}{r^3} +6\frac{\phi}{r^4}\bigg)  \left( u^4-\frac13 u^6+O(u^{8})\right)
			-   2\dfrac{\phi}{r^4}  \left(  u^4-u^6+O(u^{8}) \right) \bigg]\\
	&+\frac{\alpha^2}{2}    \int_{0}^\infty  \left( \frac{\psi'}{r^4}  -4  \frac{\psi}{r^5}    \right) \left(  u^4-\frac23 u^6+O(u^{8}) \right),
	\end{aligned}
\]

and
\begin{equation}\label{eq:he}
\begin{aligned}
	H_5=&{}~  \gamma_S \alpha^2\int_{0}^\infty \left(\frac{\phi'}{r^2}-2 \frac{\phi}{r^3} \right)   \sin(2u) u_r u^2 .
\end{aligned}
\end{equation}
Regrouping terms of the same order, we get
\[
\begin{aligned}
	H_2=&\int_{0}^\infty    \bigg[ \left( \gamma_S \bigg[\frac{1}{2} \phi'' - \left( \frac{\phi}{r} \right)_r -2  \dfrac{ \phi}{r^2} \bigg]+ \frac{\psi'}{r^2}  -2 \frac{\psi}{r^3} \right) u^2  
			+\frac23  \left( \frac{\psi}{r^3} - \frac{\psi'}{2r^2} + 2\gamma_S \dfrac{ \phi}{r^2} \right) u^4 \bigg]\\
			&+\frac2{45}\int_{0}^\infty  \left( \frac{\psi'}{r^2}  -2 \frac{\psi}{r^3}
			- \gamma_S  6 \dfrac{ \phi}{r^2} \right) \left( u^6+O(u^8) \right).
\end{aligned}
\]
Similarly,  for $H_3$ we get
\[
\begin{aligned}		
	H_3=&- \alpha^2 \int_{0}^\infty  
							 \left(\frac{\psi'}{r^2}  -2 \frac{\psi}{r^3}  
							- 4 \gamma_S  \frac{ \phi}{r^2} \right) u^2  (u_t^2+u_r^2) \\
		&- \alpha^2 \int_{0}^\infty  
						\bigg[
							\frac13 \left( 2 \frac{\psi}{r^3} - \frac{\psi'}{r^2} 
							+6  \gamma_S  \frac{ \phi}{r^2} \right) u^4 \\
							 &\quad \quad \quad \quad \quad \quad +\frac2{45} \left(\frac{\psi'}{r^2}  -2 \frac{\psi}{r^3} 
							- 8 \gamma_S   \frac{ \phi}{r^2}\right) \left( u^6+O(u^8) \right) 
						\bigg] (u_t^2+u_r^2) .
%	H_4=&- \alpha^2 \int_{0}^\infty  
%						\left[
%							\left(\frac{\psi'}{r^2}   - 2\frac{\psi}{r^3} \right)\left( u^2-\frac13 u^4+\frac2{45} u^6+O(u^8)\right)
%							+b  \frac{ \phi}{r^2} \left( 4u^2-2u^4+\frac{16}{45}u^6+O(u^8) \right) 
%							\right] u_r^2\\
\end{aligned}
\]	
%and
%\[
%\begin{aligned}		
%	H_4=&- \alpha^2 \int_{0}^\infty  
%							 \left(\frac{\psi'}{r^2}  -2 \frac{\psi}{r^3}  
%							- 4b  \frac{ \phi}{r^2} \right) u^2  u_r^2 \\
%		&- \alpha^2 \int_{0}^\infty  
%						\left[
%							 \left(-\frac13 \frac{\psi'}{r^2}  +\frac23  \frac{\psi}{r^3} 
%							+2 b  \frac{ \phi}{r^2} \right) u^4 
%							 +\left(\frac2{45}\frac{\psi'}{r^2}  -\frac4{45} \frac{\psi}{r^3} 
%							- b\frac{16}{45}  \frac{ \phi}{r^2}\right) \left( u^6+O(u^8) \right) 
%						\right] u_r^2.
%\end{aligned}
%\]	
For $H_4$, we have	
	\[
\begin{aligned}							
	H_4=
	&{}~\alpha^2    \int_{0}^\infty  \left(\frac12 \frac{\psi'}{r^4}  -2  \frac{\psi}{r^5}   -  2 \gamma_S\dfrac{\phi}{r^4} +  \gamma_S\frac{\phi''}{r^2} -4 \gamma_S\frac{\phi'}{r^3} +6 \gamma_S\frac{\phi}{r^4}  \right)  u^4 \\
	&+\frac13  \alpha^2\int_{0}^\infty \left( - \gamma_S\frac{\phi''}{r^2} +4  \gamma_S\frac{\phi'}{r^3} -6 \gamma_S\frac{\phi}{r^4}
			+ 6 \gamma_S\dfrac{\phi}{r^4}   
	 - \frac{\psi'}{r^4}  +4   \frac{\psi}{r^5}    \right) \left( u^6+O(u^{8}) \right).
\end{aligned}
\]
For $H_5$, first  we note
\[
 \sin(2u) u_r u^2
		=  \frac14 \frac{d}{dr} (2u \sin(2u)-2u^2 \cos(2u)+\cos(2u)-1).
\]
Now, replacing in $H_5$ and integrating by parts, we get
\begin{equation*}
\begin{aligned}
	H_5
%		=&{}~ b \alpha^2\int_{0}^\infty \left(\frac{\phi'}{r^2}-2 \frac{\phi}{r^3} \right)   \sin(2u) u_r u^2\\
%		=&{}~ b \alpha^2\int_{0}^\infty \left(\frac{\phi'}{r^2}-2 \frac{\phi}{r^3} \right)   \frac14 \frac{d}{dr} (2u \sin(2u)-2u^2 \cos(2u)+\cos(2u)-1)\\
		=& -\frac{\gamma_S}{4} \alpha^2\int_{0}^\infty \left(\frac{\phi'}{r^2}-2 \frac{\phi}{r^3} \right)_r     (2u \sin(2u)-2u^2 \cos(2u)+\cos(2u)-1),
		\end{aligned}
\end{equation*}
		using its Taylor expansion and regrouping terms, we have
		\begin{equation*}
\begin{aligned}
			H_5
			=& -\frac{\gamma_S}{4} \alpha^2\int_{0}^\infty \left(\frac{\phi''}{r^2}-4 \frac{\phi'}{r^3} +6\frac{\phi}{r^4}\right)   \left(2u^4-\frac89 u^6+O(u^8) \right)\\
		=&{}~ 
		 \alpha^2 \frac{\gamma_S}{2} \left\{ \int_{0}^\infty \left(4 \frac{\phi'}{r^3}-\frac{\phi''}{r^2} -6\frac{\phi}{r^4}\right)  u^4
		+  \frac{2}{9}   \int_{0}^\infty \left(\frac{\phi''}{r^2}-4 \frac{\phi'}{r^3} +6\frac{\phi}{r^4}\right)   \left( u^6+O(u^8) \right)\right\}.
		\medskip
\end{aligned}
\end{equation*}
{}
Collecting the last equation and $H_4$, we obtain \medskip
	\begin{equation}\label{eq:he1}
\begin{aligned}							
	H_4+H_5
%	=&{}~\alpha^2    \int_{0}^\infty  \left(\frac12 \frac{\psi'}{r^4}  -2  \frac{\psi}{r^5}   -  2b\dfrac{\phi}{r^4} + b\frac{\phi''}{r^2} -4b\frac{\phi'}{r^3} +6b\frac{\phi}{r^4} -\frac b2\frac{\phi''}{r^2}+2b \frac{\phi}{r^3} -3b\frac{\phi}{r^4} \right)  u^4 \\
%	%
%	&+\alpha^2\int_{0}^\infty \left( -\frac13 b\frac{\phi''}{r^2} +\frac43 b\frac{\phi'}{r^3} -2b\frac{\phi}{r^4}
%			+ 2b\dfrac{\phi}{r^4}   
%	 -\frac13 \frac{\psi'}{r^4}  +\frac43   \frac{\psi}{r^5}   + \frac29 b \frac{\phi''}{r^2}-\frac89 b \frac{\phi}{r^3} +\frac43 b\frac{\phi}{r^4}\right) \left( u^6+O(u^{8}) \right)\\
	 =
	&{}~\alpha^2    \int_{0}^\infty  \left(\frac12 \frac{\psi'}{r^4}  -2  \frac{\psi}{r^5}   +\gamma_S \dfrac{\phi}{r^4} -2\gamma_S \frac{\phi'}{r^3}  +\frac{\gamma_S}{2}\frac{\phi''}{r^2}  \right)  u^4 \\
	&+\alpha^2\int_{0}^\infty \left( 
							-\frac19 \gamma_S \frac{\phi''}{r^2} +\frac49 \gamma_S \frac{\phi'}{r^3} 
							 -\frac13 \frac{\psi'}{r^4}  +\frac43   \frac{\psi}{r^5}   
							 +\frac43 \gamma_S \frac{\phi}{r^4}
	 					\right) \left( u^6+O(u^{8}) \right).
\end{aligned}
\end{equation}\medskip
Having in mind that $\psi=r\phi$, we have
\[
%\phi=\frac{\psi}{r}
\phi'=\frac{\psi'}{r}-\frac{\psi}{r^2}\quad  \text{ and }\quad \phi''=\frac{\psi''}{r}-2\frac{\psi'}{r^2}+2\frac{\psi}{r^3}.
\]
Now, rewriting $H_i$, for $i=1,\dots,5 $, in terms of $\psi$ and its derivatives, we get

\begin{equation}\label{eq:H1}
\begin{aligned}
H_1
	=&
			-\frac12 \int_{0}^\infty  \left( \psi' -2 \gamma_S \frac{\psi}{r} \right)  u_t^2  
			-\frac12 \int_{0}^\infty  \left( \psi' +(2 \gamma_S  -4) \frac{\psi}{r}\right) u_r^2 , 
							\end{aligned}
				\end{equation}		
				
\begin{equation}\label{eq:H2}
\begin{aligned}	
H_2=&
					\int_{0}^\infty    \left( (\gamma_S-2) \frac{\psi}{r^3} +(1-2\gamma_S)  \frac{\psi'}{r^2} + \frac{\gamma_S}{2} \frac{\psi''}{r}  \right) u^2  
			-\int_{0}^\infty  \left( \frac{\psi'}{3r^2}  -\frac23(1+2\gamma_S) \frac{\psi}{r^3}\right) u^4\\
			&+\int_{0}^\infty  \left( \frac2{45}\frac{\psi'}{r^2}  -\frac4{45}\left( 1+\gamma_S \right) \frac{\psi}{r^3}
			 \right) \left( u^6+O(u^8) \right),
			 				\end{aligned}
				\end{equation}		
				
\begin{equation}\label{eq:H3}
\begin{aligned}	
H_3=&
						 \alpha^2 \int_{0}^\infty  
							 \left( (2+4\gamma_S) \frac{\psi}{r^3}   -\frac{\psi'}{r^2} \right) u^2  (u_t^2+u_r^2)  \\
		&+ \alpha^2 \int_{0}^\infty  
						\bigg[
							 \left(\frac13 \frac{\psi'}{r^2}  -\frac23 (1+6\gamma_S) \frac{\psi}{r^3}  \right) u^2 \\
							 &\quad \quad \quad \quad \quad \quad \quad -\left(\frac2{45}\frac{\psi'}{r^2}  -\frac4{45}(1+4\gamma_S) \frac{\psi}{r^3} \right) \left( u^4+O(u^6) \right) 
						\bigg]u^2 (u_t^2+u_r^2),
				\end{aligned}
				\end{equation}		
and			
\begin{equation}\label{eq:H45}
\begin{aligned}			
H_4+H_5 =
	&{}~\alpha^2    \int_{0}^\infty  
						\left(\frac12(1-6\gamma_S) \frac{\psi'}{r^4}  -(2-4\gamma_S)  \frac{\psi}{r^5}  +\frac{\gamma_S}{2}  \frac{\psi''}{r^3}  \right)  u^4 \\
	&+\alpha^2\int_{0}^\infty \left( 
							 \left(\frac69 \gamma_S  +  \frac43  \right) \frac{\psi}{r^5}
							 +\left( \frac69 \gamma_S -\frac13   \right) \frac{\psi'}{r^4}
							 -\frac19 \gamma_S \frac{\psi''}{r^3}
	 					\right) \left( u^6+O(u^{8}) \right).
\end{aligned}
\end{equation}

Finally, collecting \eqref{eq:H1}, \eqref{eq:H2}, \eqref{eq:H3}, \eqref{eq:H45}, and regrouping terms of the same order, we obtain

\begin{equation*}
\begin{aligned}
\frac{d}{dt}\Hs =
			& -\frac12 \int_{0}^\infty  \left( \psi' -2 \gamma_S \frac{\psi}{r} \right)  u_t^2 
			-\frac12 \int_{0}^\infty  \left( \psi' +(2 \gamma_S  -4) \frac{\psi}{r}\right) u_r^2  \\
					& -\int_{0}^\infty    \left( (2-\gamma_S) \frac{\psi}{r} +(2\gamma_S-1)  \psi' - \frac{\gamma_S}{2} r \psi'' \right) \frac{u^2}{r^2}\\  
						& 	 - \alpha^2 \int_{0}^\infty  
							 \left( \psi'  -(2+4\gamma_S) \frac{\psi}{r}   \right) \frac{u^2}{r^2}  (u_t^2+u_r^2) \\						&-{}~\alpha^2    \int_{0}^\infty  
						\left(-\frac12(1-6\gamma_S) \psi'  +(2-4\gamma_S)  \frac{\psi}{r}  -\frac{\gamma_S}{2}  r \psi''  \right)  \frac{u^4}{r^4}
+H_e(t),
\end{aligned}
\end{equation*}
where
\begin{equation*}
\begin{aligned}
H_e&(t) =
		\frac19\alpha^2\int_{0}^\infty \left( 
							\left( 6 \gamma_S  +  12  \right) \frac{\psi}{r}
							+\left( -3   +6 \gamma_S  \right) \psi'
							- \gamma_S r\psi''
	 					\right) \left(\frac{u^6}{r^4}+\frac{O(u^8)}{r^4}\right)\\	
						&-\frac13 \int_{0}^\infty  \left( \psi'  -2(1+2\gamma_S) \frac{\psi}{r}\right) \frac{u^4}{r^2}
			+\frac2{45}\int_{0}^\infty  \left( \psi'  -2\left( 1+\gamma_S \right) \frac{\psi}{r}
			 \right) \left(\frac{u^6}{r^2}+\frac{O(u^8)}{r^2}\right) \\
			 &- \frac{\alpha^2}{45} \int_{0}^\infty  
						\left[
							 15\left( 2 (1+6\gamma_S) \frac{\psi}{r}  - \psi' \right) 
							 +2\left( \psi'  -2(1+4\gamma_S) \frac{\psi}{r} \right) \left( u^2+O(u^4) \right) 
						\right] \frac{u^4}{r^2}(u_t^2+u_r^2).	
\end{aligned}
\end{equation*}
This ends the proof.
\end{proof}

		 Under the hypothesis of Lemma \ref{lem:HS_psi}, the functionals $\Ks$ and $\Ps$ (see \eqref{K}--\eqref{P}) are well-defined. In fact, using the Cauchy-Schwarz inequality,  we have
	\[
	\begin{aligned}
	&\Ks \leq \int_{0}^\infty \psi \left(  1+\frac{2 \alpha^2 \sin^2(u)}{r^2}\right)(u_t^2+u_r^2),
	\end{aligned}
		\]
		\noindent
	and	
		\[
	\begin{aligned}
	&	\Ps \leq  \int_{0}^\infty r\phi  \left(1+\frac{2 \alpha^2 \sin^2(u)}{r^2}\right)(u_t^2+\frac{u^2}{r^2}).
	\end{aligned}
		\]
		Then, assuming $\psi = r \phi$, and $u\in\mathcal{E}_{0}^{S,\psi}$, we get
			\[
			\begin{aligned}
	|\Ks|+|\Ps | \leq  E_{S,\psi}[u](t),
	\end{aligned}
	\]
		hence concluding that the functionals $\Ks$ and $\Ps$ in \eqref{K}-\eqref{P} are well-defined.

	\begin{cor}
	Let $n ,\gamma_S\in\R$
	 and $\psi=r^n \chi $. Then, under the hypothesis of Lemma \ref{lem:HS_psi},  the following identity holds:
		
	\begin{equation*}
\begin{aligned}
\frac{d}{dt}&\Hs =\\
			&  -\frac12 \int_{0}^\infty  \left( (n-2\gamma_S)r^{n-1}\chi+r^n\chi' \right)  u_t^2  
			-\frac12 \int_{0}^\infty   \left( (n+2\gamma_S-4)r^{n-1}\chi +r^n\chi' \right)u_r^2 , \\
					& -\int_{0}^\infty     \left( (2-\gamma_S+n(2\gamma_S-1)-\frac{\gamma_S}{2}n(n-1))r^{n-1}\chi +(\gamma_S(2-n)-1)r^n\chi' -\frac{\gamma_S}{2}r^{n+1}\chi''\right) \frac{u^2}{r^2}\\  					
	&	-\alpha^2    \int_{0}^\infty  
						\left( \frac12 (4 - n -  8\gamma_S -\gamma_S (-7 + n) n)r^{n-1}\chi +(\gamma_S(3-n)-\frac12)r^n\chi' -\frac{\gamma_S}{2}r^{n+1}\chi''\right) \frac{u^4}{r^4} \\
						&  - \alpha^2 \int_{0}^\infty  
							\left( (n-2-4\gamma_S)r^{n-1}\chi +r^n\chi' \right) \frac{u^2}{r^2}  (u_t^2+u_r^2) 
						+H_e(t)
\end{aligned}
\end{equation*}	
and for $H_e$ holds
\begin{equation*}
\begin{aligned}
H_e(t)=&
			-\frac13 \int_{0}^\infty  \left( (n  -2(1+2\gamma_S) ) r^{n-1}\chi+r^n \chi' \right) \frac{u^4}{r^2}\\
			&+\frac2{45}\int_{0}^\infty  \left( (n  -2\left( 1+\gamma_S \right))r^{n-1}\chi+r^{n}\chi'
			 \right) \left(\frac{u^6}{r^2}+\frac{O(u^8)}{r^2}\right) \\
			 &- \alpha^2 \int_{0}^\infty  
						\left[
							 \frac13\left((- n  +2 (1+6\gamma_S))r^{n-1}\chi+r^n \chi'  \right) u^2
						\right] \frac{u^2}{r^2} (u_t^2+u_r^2)\\	
			 &-  \frac2{45} \alpha^2 \int_{0}^\infty  
							\left[ (n  -2(1+4\gamma_S))r^{n-1\chi} +r^n \chi' \right] \left( u^4+O(u^6) \right) 
						 \frac{u^2}{r^2}(u_t^2+u_r^2) \\	
						 &+\frac19\alpha^2\int_{0}^\infty 
						\big[ (\gamma_S n(5-n)-3n+6\gamma_S +12)r^{n-1}\chi' \\
						&\quad \quad \quad\quad \quad \quad\quad \quad \quad+(\gamma_S(6-n)-3)r^n\chi' -\gamma_S r^{n+1}\chi''\big]
						 \left(\frac{u^6}{r^4}+\frac{O(u^8)}{r^4}\right).\\	
\end{aligned}
\end{equation*}
	\end{cor} 
\begin{proof}
The proof follows directly replacing $\psi=r^n \chi$ in Lemma \ref{lem:HS_psi}.
\end{proof}

\medskip
Now, if we set $ \chi= 1$, we obtain the following result:

\begin{cor}\label{lem:Hs_psi}
Let $\delta>0$ small enough, $\psi=r^n$  and   $u$ 
be a global solution of \eqref{skyrme} such that $u\in \mathcal{E}_{0}^{S,r^n}$ and $\|u\|_{L^{\infty}}< \delta$. 
Then, for $\gamma_S=-1$ and $n\geq 2$,
 the functional $\Hs$  \eqref{eq:HS} satisfies
				\begin{equation*}\label{eq:HS_identity}
				\begin{aligned}
		\dfrac{d}{dt} \Hs =&
 	 -\frac12 \int_{0}^\infty  r^{n-1}\left[ (n+2) u_t^2  +  (n-6) u_r^2 +( n-6)(n-1) \frac{u^2}{r^2} \right.	\\
						&\quad \quad \quad \quad \quad \ \ \ \ \left. +2\alpha^2  \left( n+2\right) \frac{u^2}{r^2}  (u_t^2+u_r^2)	
						+\alpha^2   (n-6)(n-2)  \frac{u^4}{r^4}  \right]  
						+H_e(t),
\end{aligned}
\end{equation*}
with $|H_e(t)|\leq \delta^2 E_{S,r^{n-1}}(t)$.
\end{cor}

Assuming $\delta>0$ small enough, $n\geq 6$ and applying Corollary \ref{lem:Hs_psi}, 
we obtain  the following virial inequality
				\begin{equation}\label{eq:HS_positividad}
				\begin{aligned}
		-\dfrac{d}{dt} \Hs \geq &{}~
 	 \frac14 \int_{0}^\infty  r^{n-1}\left[ (n+2) u_t^2  +  (n-6) u_r^2 +( n-6)(n-1) \frac{u^2}{r^2} \right.	\\
						&\quad \quad \quad \quad \quad \ \ \ \ \left. +2\alpha^2  \left( n+2\right) \frac{u^2}{r^2}  (u_t^2+u_r^2)	
						+\alpha^2   (n-6)(n-2)  \frac{u^4}{r^4}  \right] 
						\geq 0.
\end{aligned}
\end{equation}
In particular, as an application of \eqref{eq:HS_positividad},  we obtain the following  result for $r^{6+\epsilon}$ and $r^{7+\epsilon}$ weighted energies.
\begin{cor}\label{cor:Sint_local}
	Let $\epsilon>0$ and $u$ be a small global solution of \eqref{skyrme} in the class $\mathcal{E}_{0}^{S,r^{6+\epsilon}}\cap \mathcal{E}_{0}^{S,r^{7+\epsilon}}$. 
	Then, 
	\begin{enumerate}
		\item Integrability in time:
		\begin{equation*}\label{eqn:S_seq1}
		\begin{aligned}
		\int_{2}^{\infty} \int_{0}^{\infty}   
								(r^{6+\epsilon} +r^{7+\epsilon})\left[  u_t^2  +   u_r^2 +  \frac{u^2}{r^2} 	
						+2\alpha^2 \frac{u^2}{r^2}  (u_t^2+u_r^2)	
						+\alpha^2 \frac{u^4}{r^4}  \right] 
		drdt \lesssim_{\epsilon} 1.
		\end{aligned}
		\end{equation*}
		\item Sequential decay to zero: there exists $s_n, \ t_n\uparrow \infty$ such that
		\begin{equation}\label{skyrme_sucesion}
		\lim_{n\to \infty} E_{S,r^{6+\epsilon}}[u](t_n) =0\  \mbox{ and }\ \lim_{n\to \infty} E_{S,r^{7+\epsilon}}[u](s_n) =0.
		\end{equation}
	\end{enumerate}
\end{cor}	
\begin{proof}
The proof of the first statement  is obtained directly integrating \eqref{eq:HS_positividad}.\\
The second statement follows from the approximation of the  energy $E_{S,\phi}[u](t)$ in \eqref{energy_S} for $\|u\|_{L^{\infty}}$ small enough. Then, using \eqref{eq:taylor_S}, we get
\[
\begin{aligned}
	E_{S,\phi}[u](t)
%	=&
%	\int_{0}^{\infty} \phi(r) \left[ u_t^2+u_r^2+ 2 \alpha^2 \dfrac{u^2}{r^2} (u_t^2+u^2_r)+ 2\dfrac{u^2}{r^2}
%	+\alpha^2 \dfrac{ u^4}{r^4} \right]\\
%	&+ 
%	\int_{0}^{\infty} \phi(r) \left[  2 \alpha^2 \dfrac{O(u^4)}{r^2} (u_t^2+u^2_r)+ 2\dfrac{O(u^4)}{r^2}
%	+\alpha^2 \dfrac{ u^4}{r^4} \right]\\
	\lesssim &
	\int_{0}^{\infty} \phi(r) \left[ u_t^2+u_r^2+ 2 \alpha^2 \dfrac{u^2}{r^2} (u_t^2+u^2_r)+ 2\dfrac{u^2}{r^2}
	+\alpha^2 \dfrac{ u^4}{r^4} \right]\\
	&+ 
	\|u\|^2_{L^\infty}\int_{0}^{\infty} \phi(r) \left[  2 \alpha^2 \dfrac{u^2}{r^2} (u_t^2+u^2_r)+ 2\dfrac{u^2}{r^2}
	+\alpha^2 \dfrac{ u^2}{r^4} \right]\\
		\lesssim_{\delta} &
	\int_{0}^{\infty} \phi(r) \left[ u_t^2+u_r^2+ 2 \alpha^2 \dfrac{u^2}{r^2} (u_t^2+u^2_r)+ 2\dfrac{u^2}{r^2}
	+\alpha^2 \dfrac{ u^4}{r^4} \right].
	\end{aligned}
	\]
	We conclude replacing $\phi=r^{6+\epsilon}$ and using the first statement.
\end{proof}
 	
	\subsubsection{Decay to zero for modified Energy: Proof of the Theorem \ref{thm:S_interior} for the Adkins-Nappi model}
	
Similarly to the Skyrme equation, we will need the following technical lemmas.
 
\begin{lem}\label{lem:HAN_psi}
Let $u$ be a global solution of \eqref{AdkinsNappi} such that $u\in \mathcal{E}_{0}^{AN,\psi}$, and $\psi=r \phi$ 
(where $\psi$ and $\phi$ are the weight functions presented in \eqref{momentum} and \eqref{R}).
Then, the functional $\Han$, defined  in  \eqref{eq:Han}, satisfies the following identity
				\begin{equation}\label{eq:Han'}
		\begin{aligned}
		\dfrac{d}{dt}&\Han
				=\\
				&-\frac12 \int_{0}^{\infty} \left(\psi' -2\gamma_{AN}\frac{\psi}{r} \right)  u_t^2  
				- \frac12 \int_{0}^{\infty}  \left( \psi' -2(\gamma_{AN}-2)\dfrac{\psi }{r} \right)u_r^2\\
				&-\frac12 \int_{0}^{\infty}  \left[ 
								 2(1-\gamma_{AN} )\frac{\psi}{r} - \psi'+ \gamma_{AN} r\psi''
							\right]	\frac{u^2}{r^2}
				+\frac13 \int_{0}^{\infty} \left[ (4\gamma_{AN}+1)  \frac{\psi}{r} -\frac12 \psi'
								\right] \frac{u^4}{r^2} \\
				&-\frac1{45}\int_{0}^{\infty} \left[
								 2(1+6\gamma_{AN})\frac{\psi}{r} -\psi'
							 \right]\left(\frac{u^6}{r^2}+\frac{O(u^8)}{r^2}\right)\\
				&-\frac19 \int_{0}^{\infty} \left[
					4(2+3\gamma_{AN}) \frac{\psi}{r} -2 \psi' 
							 \right]\left(\frac{u^6}{r^4}+\frac{O(u^8)}{r^4}\right).
		\end{aligned}
		\end{equation}
		\end{lem}

\begin{proof}
First, we note that  $\Man$ and $\Ran$ are well defined in $\mathcal{E}_{0}^{AN,\psi}$. Collecting \eqref{eq:dt_M} and \eqref{eq:dt_R}, we get that $\Han$ is given by
		{\small
		\[
		\begin{aligned}
		\dfrac{d}{dt}\Han 
		=&
			-\frac12 \int_{0}^{\infty} (\psi'-2\gamma_{AN}\phi)  u_t^2 
			- \int_{0}^{\infty}  \left(\frac{\psi'}{2} -\dfrac{2\psi }{r}\right)u_r^2-\gamma_{AN}\int_{0}^{\infty} \phi u_r^2\\
			&-\frac12 \int_{0}^{\infty} \left( 2 \frac{\psi}{r} -\psi' \right)\frac{\sin^2(u) }{r^2} 
			-\frac12\int_{0}^{\infty} \left(  4 \frac{\psi}{r} -\psi' \right) \frac{( u- \sin (u) \cos (u))^2}{r^4}\\
			&-\gamma_{AN}\int_{0}^{\infty}  \left[  \phi' r-\phi  -\frac{r^2\phi_{rr}}{2}\right]\frac{u^2}{r^2}
			-\gamma_{AN}\int_{0}^{\infty}  \left(\dfrac{\phi}{r^2}\right) u \sin(2u) \\
			&-\gamma_{AN} \int_{0}^{\infty} \left(\dfrac{\phi}{r^4}\right) u \left( u- \sin (u) \cos (u) \right) (1-\cos (2u)).
\end{aligned}
\]}
Now, let $\delta>0$ small enough and using the Taylor approximation for $\|u\|_{L^{\infty}}<\delta$, we have 
		\[
		\begin{gathered}
		u \sin(2u)=2u^2-\frac43 u^4+\frac{4}{15}u^6+O(u^8),\\
		\sin^2(u)= u^2-\frac13 u^4+\frac{2}{45}u^6+O(u^8),\\
		(u-\sin u\cos u)^2= \frac49 u^6-\frac{8}{45} u^8+O(u^{10}),\\
		u(u- \sin(u)\cos(u))(1-\cos(2u))=\frac43 u^6-\frac{32}{47}u^8+O(u^{10}).
		\end{gathered}
		\]
		Replacing in $\frac{d}{dt}\Han$ and regrouping terms of same order, we get
		\[
		\begin{aligned}
		\dfrac{d}{dt}&\Han
				=\\
				&-\frac12 \int_{0}^{\infty} (\psi'-2\gamma_{AN}\phi)  u_t^2 
				- \int_{0}^{\infty}  \left(\frac{\psi'}{2} -\dfrac{2\psi }{r}+\gamma_{AN} \phi \right)u_r^2\\
				&-\int_{0}^{\infty}  \left[ 
								 \frac{\psi}{r} -\frac12 \psi'+ \gamma_{AN}\left( \phi' r+\phi  
								 	-\frac{r^2\phi_{rr}}{2}\right)
							\right]	\frac{u^2}{r^2}
				-\int_{0}^{\infty} \left[ \frac16 \psi'-\frac13  \frac{\psi}{r} -\frac43 \gamma_{AN} \phi
								\right] \frac{u^4}{r^2} \\
				&-\frac{1}{45}\int_{0}^{\infty} \left[
								 2\frac{\psi}{r} -\psi'+12 \gamma_{AN}\phi
							 \right]\left(\frac{u^6}{r^2}+\frac{O(u^8)}{r^2}\right)\\
				&-\frac2{9}\int_{0}^{\infty} \left[
					4 \frac{\psi}{r} - \psi' +12 \gamma_{AN} \phi
							 \right]\left(\frac{u^6}{r^4}+\frac{O(u^8)}{r^4}\right).
		\end{aligned}
		\]
		Since $\psi=r\phi$, we get
		\[
		\phi'= \frac{\psi'}{r}-\frac{\psi}{r^2}, \quad 
		\phi''= \frac{\psi''}{r}-2\frac{\psi'}{r^2}+2\frac{\psi}{r^3}.
		\]
		 Then, rewriting $\Han$ in terms of $\psi$, we get
				\[
		\begin{aligned}
		\dfrac{d}{dt}&\Han
				=\\
				&-\frac12 \int_{0}^{\infty} \left(\psi'-2\gamma_{AN}\frac{\psi}{r} \right)  u_t^2 
				- \frac12 \int_{0}^{\infty}  \left( \psi' -2(\gamma_{AN}-2)\dfrac{\psi }{r} \right)u_r^2\\
				&-\frac12 \int_{0}^{\infty}  \left[ 
								 2(1-\gamma_{AN} )\frac{\psi}{r} - \psi'+ \gamma_{AN} r\psi''
							\right]	\frac{u^2}{r^2}
				-\frac13 \int_{0}^{\infty} \left[ -(4\gamma_{AN}+1)  \frac{\psi}{r} +\frac12 \psi'
								\right] \frac{u^4}{r^2} \\
				&-\frac1{45}\int_{0}^{\infty} \left[
								 2(1+6\gamma_{AN})\frac{\psi}{r} -\psi'
							 \right]\left(\frac{u^6}{r^2}+\frac{O(u^8)}{r^2}\right)\\
				&-\frac19 \int_{0}^{\infty} \left[
					4(2+3\gamma_{AN}) \frac{\psi}{r} -2 \psi' 
							 \right]\left(\frac{u^6}{r^4}+\frac{O(u^8)}{r^4}\right).
		\end{aligned}
		\]
		This ends the proof.
		\end{proof}
	
	Similarly to Skyrme equation, using  $\psi = r \phi$  and the Cauchy-Schwarz inequality,  we get
	\[
	\begin{aligned}
	|\Man|+|\Ran| \leq E_{AN,\psi}[u] (t).
	\end{aligned}
		\]
		Then, the functionals $\Man$ and $\Ran$ are well-defined if $u\in \mathcal{E}_{0}^{AN,\psi}$.

	\begin{cor}
	Under the hypothesis of Lemma \ref{lem:HAN_psi} and assuming $n ,\gamma_{AN}\in\R$
	 and $\psi=r^n \chi $,  the following holds:
		\[
		\begin{aligned}
		\dfrac{d}{dt}& \Han
				= \\
				&-\frac12 \int_{0}^{\infty} \left[ \left(n-2\gamma_{AN}\right)  r^{n-1}\chi +r^n\chi' \right]u_t^2 
				- \frac12 \int_{0}^{\infty} \left[  \left( n +4-2\gamma_{AN} \right)r^{n-1}\chi+r^n\chi'  \right]u_r^2~{}\\
				&-\frac12 \int_{0}^{\infty}  \left[ (n-2)(1-\gamma_{AN} n-\gamma_{AN})
							 r^{n-1}\chi+(\gamma_{AN}-1)r^n \chi' +\gamma_{AN} r^{n+1}\chi'' \right]\frac{u^2}{r^2}\\
				&-\frac13 \int_{0}^{\infty} \left[\left( -(4\gamma_{AN}+1)   +\frac12 n
								\right)  r^{n-1}\chi +r^n\chi' \right] \frac{u^4}{r^2} \\
				&-\frac1{45}\int_{0}^{\infty} \left[\left(
								 2(1+6\gamma_{AN}) -n
							 \right) r^{n-1}\chi -r^n\chi' \right]  \left(\frac{u^6}{r^2}+\frac{O(u^8)}{r^2}\right)\\
				&-\frac19 \int_{0}^{\infty}\left[ \left(
					4(2+3\gamma_{AN})  -2 n 
							 \right) r^{n-1}\chi-2r^n\chi' \right]  \left(\frac{u^6}{r^4}+\frac{O(u^8)}{r^4}\right).\\
		\end{aligned}
		\]		
	\end{cor} 
\begin{proof}
The proof follows directly using \eqref{eq:Han'} and replacing  $\psi=r^n\chi$.
\end{proof}

Now, considering that $\chi=1$, we obtain the following result:
\begin{cor}\label{lem:HAN_rn}
Let $\psi=r^n$ , and  $u$ be a global solution of \eqref{AdkinsNappi} such that $u\in \mathcal{E}_{0}^{AN,r^n}$. Then, for $\gamma_{AN}=(n-2)/8$ and $n\geq 2$,
the functional $\mathcal{H}_{AN}$,  defined in \eqref{eq:Han}, satisfies the following identity
				\begin{equation*}\label{eq:Han_identity}
				\begin{aligned}
		\dfrac{d}{dt} \Han\leq &
				-\frac12 \int_{0}^{\infty} 
									r^{n-1} \left( \frac{3n+2}{4} u_t^2 
				+   \frac{ 3(6+n)}{4}   u_r^2
				+ \frac{(n-2)(n^2-n-10)}{8}  \frac{u^2}{r^2}\right.\\
				&\left.  \quad \quad  \quad \quad \quad \quad \quad +\frac{n-2}{45}  \left(\frac{u^6}{r^2}+\frac{O(u^8)}{r^2}\right)
				+\frac{( 10-n )}{9}  \left(\frac{u^6}{r^4}+\frac{O(u^8)}{r^4}\right)\right).
				\end{aligned}
				\end{equation*}
\end{cor}
\medskip

For $n\in \left[\frac{1+\sqrt{41}}{2},10\right]$, by Corollary \ref{lem:HAN_rn}, we obtain the following inequality\medskip
	\begin{equation}\label{eq:Han_positividad}
				\begin{aligned}
		-\dfrac{d}{dt} \Han\geq &~{}
				\frac14 \int_{0}^{\infty} 
									r^{n-1} \left( \frac{3n+2}{4} u_t^2 
				+   \frac{ 3(6+n)}{4}   u_r^2
				+ \frac{(n-2)(n^2-n-10)}{8}  \frac{u^2}{r^2}\right.\\
				&\left.  \quad \quad  \quad \quad \quad \quad \quad +\frac{n-2}{45}  \frac{u^6}{r^2}
				+\frac{( 10-n )}{9}  \frac{u^6}{r^4}\right) 
				\geq 0,
				\end{aligned}
				\end{equation} 
				\medskip
which is essential to obtain the integrability property. In particular, we obtain the following result for the  $r^{4+\epsilon}$ and $r^{5+\epsilon}$ weighted energies.

\begin{cor}\label{cor:AN_int_local}
	Let $u$ be a global solution of \eqref{AdkinsNappi} in the class $\mathcal{E}_{0}^{AN,r^{4+\epsilon}}\ \bigcap\ \mathcal{E}_{0}^{AN,r^{5+\epsilon}}$ for $\epsilon\in[0,4[$. 
	Then, 
	\begin{enumerate}
		\item Integrability in time:
		\begin{equation*}\label{eqn:AN_seq1}
		\begin{aligned}
		\int_{2}^{\infty} \int_{0}^{\infty}   
								(r^{4+\epsilon} +r^{5+\epsilon})\left(   
								(u_t^2+u_r^2) + \frac{u^2}{r^2}+\frac{u^6}{r^4}
									 \right)
		drdt \lesssim_{u_0} 1.
		\end{aligned}
		\end{equation*}
		\item Sequential decay to zero: there exists $s_n, \ t_n\uparrow \infty$ such that
		\begin{equation*}\label{sucesion}
		\lim_{n\to \infty} E_{AN,r^{5+\epsilon}}[u](t_n) =0\  \mbox{ and }\ \lim_{n\to \infty} E_{AN,r^{4+\epsilon}}[u](s_n) =0.
		\end{equation*}
	\end{enumerate}
\end{cor}		 
	
\medskip

 The proof of above corollary follows directly from \eqref{eq:Han_positividad}. With these results, 
 we are ready to conclude the proof of Theorem \ref{thm:AN_interior} for the Skyrme and Adkins-Nappi equations.

\subsection{End of the proof of Theorem \ref{thm:S_interior}}
Consider  $E_{S,\varphi}$  as in \eqref{w_energy_S} with $\varphi= r^{7+\epsilon}$.  From \eqref{eq:I_wo_r2}, we have
\[
\dfrac{d}{dt}E_{S,\varphi}[u](t)   = -2\int_{0}^{\infty} \left(\varphi' -2\frac{\varphi}{r}\right) \left(1+\frac{2\alpha^2 \sin^2(u)}{r^2}\right) u_tu_r.
\]
Therefore,
\begin{equation*}\label{S_auxiliar}
\left|\dfrac{d}{dt}E_{S,\varphi}[u](t)  \right| \lesssim  \int_{0}^{\infty} \left|\varphi' -2\frac{\varphi}{r}\right|\left(1+\frac{2\alpha^2 \sin^2(u)}{r^2}\right)(u_t^2 +u_r^2 ).
\end{equation*}
Integrating in $[t,t_n]$, we have 
\[
\left| E_{S,\varphi}[u](t)  - E_{S,\varphi}[u](t_n) \right|\lesssim \int_{t}^{t_n} \int_{0}^{\infty}  \left|\varphi' -2\frac{\varphi}{r}\right|\left(1+\frac{2\alpha^2 \sin^2(u)}{r^2}\right)(u_r^2 +u_t^2 )drdt.
\]
Sending $n$ to infinity, we have from \eqref{skyrme_sucesion} that $ E_{S,\varphi}[u](t_n) \to 0$ and
\[
\left| E_{S,\varphi}[u](t) \right|\lesssim \int_{t}^{\infty}  \int_{0}^{\infty} \left|\varphi' -2\frac{\varphi}{r}\right|\left(1+\frac{2\alpha^2 \sin^2(u)}{r^2}\right)(u_r^2 +u_t^2 )drdt.
\]
Finally, if $t\to \infty$, we conclude. Since  $E_{S,\varphi}[u](t) \gtrsim \big\|(r^{\frac{5+\epsilon}{2}}u_t, r^{\frac{5+\epsilon}{2}}u_r)(t)\big\|_{L^2\times L^2(\R^3)}^2$, this proves Theorem \ref{thm:S_interior} for the Skyrme equation.
\medskip
The proof in the Adkins-Nappi case is analogous considering  $E_{AN,\varphi}$ in \eqref{w_energy_AN} with $\varphi= r^{5+\epsilon}$. \\

This concludes the proof of Theorem \ref{thm:AN_interior}.

\end{document}